\documentclass{amsart}
\usepackage{bbm}
\usepackage{mathrsfs}
\usepackage{cases}
\usepackage{latexsym}
\usepackage[all]{xy}
\usepackage{stmaryrd}
\usepackage{amsfonts}
\usepackage{float}
\usepackage{amsmath,amssymb,amscd,bbm,amsthm,mathrsfs,dsfont}
\usepackage{fancyhdr}
\usepackage{amsxtra,ifthen}
\usepackage{verbatim}

\numberwithin{equation}{section}

\theoremstyle{plain}
\newtheorem{prop}{Proposition}[section]
\newtheorem{thm}[prop]{Theorem}
\newtheorem{cor}[prop]{Corollary}
\newtheorem{lem}[prop]{Lemma}
\newtheorem{dfn}[prop]{Definition}

\theoremstyle{definition}
\newtheorem{example}[prop]{Example}
\newtheorem{remark}[prop]{Remark}
\newtheorem{remarks}[prop]{Remarks}

\newcommand{\lam}{\lambda}

\DeclareMathOperator{\rad}{rad}

\begin{document}

\title[Semi-simplicity of Temperley-Lieb Algebras of type D]
{Semi-simplicity of Temperley-Lieb Algebras of type D}

\thanks {Corresponding Author: Xiaolin Shi}
\thanks {Li is supported by the Natural Science Foundation of Hebei
Province, China (A2017501003) and NSFC 11871107.}

\author{Yanbo Li}

\address{Li: School of Mathematics and Statistics, Northeastern
University at Qinhuangdao, Qinhuangdao, 066004, P.R. China}

\email{liyanbo707@163.com}

\author{Xiaolin Shi}

\address{Shi: Department of mathematics, School of Science, Northeastern
University, Shenyang, 110819, P.R. China}

\email{shixiaolin0217@163.com}

\begin{abstract}
We provide a necessary and sufficient condition for a type D Temperley-Lieb algebra ${\rm TLD}_n(\delta)$
being semi-simple by studying branching rule for cell modules. As a byproduct, our result is used
to study the so-called forked Temperley-Lieb algebra, which is a quotient algebra of ${\rm TLD}_n(\delta)$.
\end{abstract}

\subjclass[2010]{16G30, 16K20, 81R05}

\keywords{(Forked) Temperley-Lieb algebra; branching rule; cellular algebra; semi-simplicity; Gram matrix.}

\maketitle

\section{Introduction}

Temperley-Lieb algebras were introduced by Temperley and Lieb \cite{TL} in 1971
in order to study the ice model and Potts model in
statistical physics. In 1983, Jones \cite{J} found these algebras again when
he studied the factor theory. Later he \cite{J2} computed the Jones polynomial of a
knot by using the Markov trace of the Temperley-Lieb algebra. Temperley-Lieb algebras play a
more and more important role in various areas. For more details, we refer the reader to \cite{A1, BK}
and the references therein, in which a Temperley-Lieb algebra is connected
to basic ideas in logic and computation, quantum mechanics and other subjects.

\smallskip

From an algebraic point of view, a Temperley-Lieb algebra is a type A Hecke algebra quotient. Along this direction,
a canonical quotient of a Hecke algebra of arbitrary finite type, which is called
a ``generalized Temperley-Lieb algebra", was introduced by Graham in his
PhD thesis \cite{G1} and independently by Fan in \cite{F} and Dieck in \cite{D}.
There are several other generalizations of the classical Temperley-Lieb algebras, such as
cyclotomic Temperley-Lieb algebras \cite{RX, RXY}, Motzkin algebras \cite{BH}, partition algebras \cite{M, X} and so on.
It is helpful to point out that these algebras are all cellular  in the sense of \cite{GL}.
Note that the theory of cellular algebras provides a systematic framework for studying
the representation theory of many important algebras, such as Hecke algebras of finite type,
Brauer algebras, Birman-Wenzl algebras and so on. We refer the reader to \cite{G, GL, X2} for details.

\smallskip

Temperley-Lieb algebras of type A have been studied deeply. For example, the quasi-heredity, semi-simplicity,
dimensions of simple modules, Jucys-Murphy elements, Grothendieck group, categorification, block theory,
Kazhdan-Lusztig basis and Murphy basis and all kinds of other related topics.
We refer the reader to \cite{A, E, F, M2, S, W, W2} for details.
Recently there has been some interesting work surrounding type D Temperley-Lieb algebras.
For example, Losonczy \cite{L2} studied the relation between its canonical basis and the Kazhdan-Lusztig basis.
Lejczyk and Stroppel \cite{LS} using these algebras described parabolic Kazhdan-Lusztig polynomials of the Weyl group
of type D in a diagrammatical way introduced by Green in \cite{G2}.
Moreover, Ehrig and Stroppel in \cite{ES} and Stroppel and Wilbert in \cite{SW} constructed
certain irreducible representations of the Temperley-Lieb algebra of type D using the cohomology of Springer fibers.

\smallskip

In this paper, we will focus on the type D Temperley-Lieb algebras employing the diagrammatic approach and the cellular framework.
More exactly, we will study the recurrence relation of Gram matrices of cell modules mainly
depending on analyzing the so-called decorated parenthesis diagrams.
As a result, we can provide a criterion of semi-simplicity for a type D Temperley-Lieb algebra
${\rm TLD}_n(\delta)$. As a byproduct, we give a criterion of semi-simplicity for a
quotient algebra of ${\rm TLD}_n(\delta)$, which is the so-called forked Temperley-Lieb algebra.
This class of algebras were introduced by Grossman in \cite{G3} in order to study intermediate subfactors.

\smallskip

The paper is organized as follows. We begin with a quick review on the general theory of cellular algebras. In Section 3,
we recall the ring theoretic definition and the diagrammatic definition of a Temperley-Lieb algebra of type D, ${\rm TLD}_n(\delta)$.
In Section 4, we generalize the parenthesis diagrams to decorated ones and give some necessary
investigation on them, especially about an order and some mappings related to it.
In Section 5, using the decorated parenthesis diagrams, we construct a cellular basis and then study branching rule for cell modules.
In Section 6, we first prove the recurrence relation on the determinants of Gram matrices of cell modules and then
deduce a necessary and sufficient condition for ${\rm TLD}_n(\delta)$ being semi-simple.
In Section 7, as a byproduct, we
study a certain quotient algebra of ${\rm TLD}_n(\delta)$, the forked Temperley-Lieb algebra, by using the main result obtained in Section 6.

\medskip

\section{Preliminaries on cellular algebra}

In this section, we give a quick review on the definition and
well-known results about a cellular algebra. The main reference is \cite{GL}.

\begin{dfn}[\protect{\cite[Definition~1.1]{GL}}]\label{2.1}
Let $R$ be a commutative ring with identity. An associative unital
$R$-algebra $A$ is called a cellular algebra with cell datum $(\Lambda, M, C, \ast)$ if the
following conditions are satisfied:

\begin{enumerate}
\item[(C1)] The finite set $\Lambda$ is a poset with
order relation $<$. Associated with
each $\lam\in\Lambda$, there is a finite set $M(\lam)$. The
algebra $A$ has an $R$-basis $\{C_{S,T}^\lam \mid S,T\in
M(\lam),\lam\in\Lambda\}$.

\item[(C2)] The map $\ast$ is an $R$-linear anti-automorphism of $A$
such that $(C_{S,T}^\lam)^{\ast}= C_{T,S}^\lam $ for
all $\lam\in\Lambda$ and $S,T\in M(\lam)$.

\item[(C3)] Let $\lam\in\Lambda$ and $S,T\in M(\lam)$. For any
element
$a\in A$, we have
$$aC_{S,T}^\lam\equiv\sum_{S^{'}\in
M(\lam)}r_{a}(S',S)C_{S^{'},T}^{\lam} \,\,\,\,\rm {mod}\,\,\,
A(<\lam),$$ where $r_{a}(S^{'},S)\in R$ is independent of $T$ and
$A(<\lam)$ is the $R$-submodule of $A$ generated by
$\{C_{U,V}^\mu \mid U,\,\,V\in M(\mu),\,\,\mu<\lam\}$.
\end{enumerate}
\end{dfn}

Let $\lam\in\Lambda$. For arbitrary elements $S,T,U,V\in M(\lam)$,
Definition~\ref{2.1} implies that
$$C_{S,T}^\lam C_{U,V}^\lam \equiv\Phi(T,U)C_{S,V}^\lam\,\,\,\, \rm mod\,\,\, A(<\lam),$$
where $\Phi(T,U)\in R$ depends only on $T$ and $U$. It is easy to check that $\Phi(T,U)=\Phi(U,T)$
for arbitrary $T,U\in M(\lam)$. For $\lam\in \Lambda$, fix an order
on $M(\lam)$. The associated Gram matrix $G(\lam)$ is the following symmetric matrix
$$G(\lam)=(\Phi(S, T))_{S,T\in M(\lam)}.$$
 Note that $\det G(\lam)$, the determinant of $G(\lam)$, is independent of the choice of the order on
$M(\lam)$.

\smallskip

Given a cellular algebra $A$, we note that $A$ has a family of modules defined by its cellular structure.

\begin{dfn}[\protect{\cite[Definition 2.1]{GL}}]\label{2.2}
Let $A$ be a cellular algebra with cell datum $(\Lambda, M, C,
\ast)$. For each $\lam\in\Lambda$, the cell module $W(\lam)$ is an $R$-module with basis
$\{C_{S}\mid S\in M(\lam)\}$ and the left $A$-action
defined by
$$aC_{S}=\sum_{S^{'}\in M(\lam)}r_{a}(S^{'},S)C_{S^{'}}
\,\,\,\,(a\in A,\,\,S\in M(\lam)),$$ where $r_{a}(S^{'},S)$ is the
element of $R$ defined in Definition~\ref{2.1}{\rm(C3)}.
\end{dfn}

For $\lam\in\Lambda$, the bilinear form $\Phi_{\lam}$ defined below plays an important role for
studying the structure of $W(\lam)$.

\begin{dfn}\label{bf}
For a cell module $W(\lam)$, define a bilinear form $$\Phi
_{\lam}:\,\,W(\lam)\times W(\lam)\longrightarrow R$$ by $\Phi
_{\lam}(C_{S},C_{T})=\Phi(S,T)$.
\end{dfn}

The following simple property of $\Phi_{\lam}$ was provided in \cite{GL}.
\begin{lem}\cite[Proposition 2.4]{GL}\label{bl}
Let $x, y\in W(\lam)$ and $a\in A$. Then $$\Phi_{\lam}(a^\ast x, y)=\Phi_{\lam}(x, ay).$$
\end{lem}

Define
$\rad\lam:= \{x\in W(\lam)\mid \Phi_{\lam}(x,y)=0
\,\,\,\text{for all} \,\,\,y\in W(\lam)\}.$ If $\Phi
_{\lam}\neq 0$, then $\rad\lam$ is the radical of the $A$-module
$W(\lam)$.

For $R$ being a field, Graham and Lehrer \cite{GL} proved the following results.
\begin{lem} {\rm\cite[Theorem 3.4]{GL}}\label{2.3}
For any
$\lam\in\Lambda$, denote the $A$-module $W(\lam)/\rad \lam$ by $L_{\lam}$. Let
$\Lambda_{0}=\{\lam\in\Lambda\mid \Phi_{\lam}\neq 0\}$. Then
$\{L_{\lam}\mid \lam\in\Lambda_{0}\}$ is a complete set of
{\rm (}representative of equivalence classes of {\rm )} absolutely simple
$A$-modules.
\end{lem}

Note that if $\Lambda=\Lambda_0$, then $A$ is quasi-hereditary in the sense of \cite{CPS}.

\begin{lem}{\rm\cite{GL}}\label{2.4}
The following are equivalent
\begin{enumerate}
\item[(1)]\, The algebra $A$ is semi-simple;
\item[(2)]\, The nonzero cell modules $W(\lam)$ are irreducible and
pairwise inequivalent;
\item[(3)]\, The form $\Phi_{\lam}$ is non-degenerate {\rm (}i.e. $\rad\lam=0${\rm )} for each $\lam\in\Lambda$;
\item[(4)]\, The determinant of Gram matrix $G(\lam)$ is nonzero for each $\lam\in\Lambda$.

\end{enumerate}
\end{lem}

\medskip

\section{Definition of Temperley-Lieb algebras of type D}

In this section, we recall the ring theoretic
definition and the diagrammatic definition of the
Temperley-Lieb algebra of type D. Since we adopt diagrammatic approach to study the algebra,
the focus is the graphical definition. The main references are \cite{F, G1} and \cite{G2}.

\subsection{Ring theoretic definition} In order to give the definition of a
Temperley-Lieb algebra of type D by generators and relations,  we first recall
the Dynkin diagram of type $D_n$.

\begin{figure}[H]
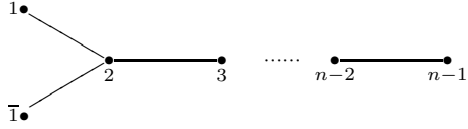

	\[
	\xy
    (1,-7)*{\scriptstyle \overline{1} \bullet}="1s";
    (1,7)*{\scriptstyle 1\bullet}="2s";
    (13,0)*{\scriptstyle\bullet}="3s";
    (13,-2)*{\scriptstyle 2};
    (28,0)*{\scriptstyle\bullet}="4s";
    (28,-2)*{\scriptstyle 3};
    (36,0)*{\scriptstyle\cdots\cdots};
    (43,0)*{\scriptstyle\bullet}="5s";
    (43,-2)*{\scriptstyle n-2};
    (58,0)*{\scriptstyle\bullet}="6s";
    (58,-2)*{\scriptstyle n-1};
     "1s"; "3s" **\dir{-};
     "2s"; "3s" **\dir{-};
     "3s"; "4s" **\dir{-};
     "5s"; "6s" **\dir{-};
    \endxy
	\]
    \caption{Dynkin diagram of type D}
	\label{Fig1}
\end{figure}

\begin{dfn}\label{3.1}
Let $R$ be a commutative ring with identity and let $\delta\in R$. For a
positive integer $n \geq 4$, the type $D$ Temperley-Lieb algebra ${\rm TLD}_n(\delta)$ is generated by elements 1 and
$e_{\bar{1}}, e_1, e_2, \cdots, e_{n-1}$ subject to the following relations
\begin{enumerate}
\item[(1)]\,$e_i^2=\delta e_i \,\,\,\,\,\forall\, i=\overline{1}, 1, 2, \cdots, n-1,$
\item[(2)]\,$e_ie_j=e_je_i\,\,\,\, \text{if $i$ and $j$ are not connected in the graph},$
\item[(3)]\,$e_ie_je_i=e_i\,\,\,\, \text{if $i$ and $j$ are connected in the graph}.$
\end{enumerate}
\end{dfn}

A word $w$ is a monomial in the generators $e_{\bar{1}}, e_1, \cdots, e_{n-1}$. It is called reduced if
the number of $e_i$ in the expression is minimal. All of the reduced word form a basis
of ${\rm TLD}_n(\delta)$ and
$$\dim {\rm TLD}_n(\delta)=\frac{n+3}{2(n+1)}\left(\begin{array}{c} 2n \\ n \\ \end{array} \right)-1.$$

\subsection{Diagrammatic definition}

As far as we know, the diagrammatic approach to
classical Temperley-Lieb algebras is due to Kauffman \cite{K}.
The algebra ${\rm TLD}_n(\delta)$ can also be considered as a diagram algebra.
The diagrammatic realization of ${\rm TLD}_n(\delta)$ was first given by Green in \cite{G2}.

\smallskip

Corresponding to the generators $e_i$, the ``diagrammatic generators" are
\begin{figure}[H]
	\[
	\xy
    (-15,0)*{\scriptstyle\circ}="1s";
    (-10,2.5)*{\scriptstyle\bullet};
    (-5,0)*{\scriptstyle\circ}="2s";
    (5,0)*{\scriptstyle\circ}="3s";
    (15,0)*{\scriptstyle\cdots\cdots};
    (25,0)*{\scriptstyle\circ}="4s";
    (35,0)*{\scriptstyle\circ}="5s";
    (-15,10)*{\scriptstyle\circ}="1d";
    (-10,7.3)*{\scriptstyle\bullet};
    (-5,10)*{\scriptstyle\circ}="2d";
    (5,10)*{\scriptstyle\circ}="3d";
    (15,10)*{\scriptstyle\cdots\cdots};
    (25,10)*{\scriptstyle\circ}="4d";
    (35,10)*{\scriptstyle\circ}="5d";
    "1s"; "2s" **\crv{(-14,3.5) & (-6,3.5)};
    "1d"; "2d" **\crv{(-14,6.5) & (-6,6.5)};
    "3s"; "3d" **\dir{-};
    "4s"; "4d" **\dir{-};
    "5s"; "5d" **\dir{-};
    \endxy
    \]
    \caption{$e_{\bar{1}}$}
	\label{Fig2}
\end{figure}
\noindent which can be considered as $e_{\bar{1}}$, and

\begin{figure}[H]
	\[
	\xy
    (-45,0)*{\scriptstyle\circ}="1s";
    (-35,0)*{\scriptstyle\cdots\cdots};
    (-25,0)*{\scriptstyle\circ}="2s";
    (-15,0)*{\scriptstyle\circ}="3s";
    (-15,-2)*{\scriptstyle i};
    (-5,0)*{\scriptstyle\circ}="4s";
    (-5,-2)*{\scriptstyle i+1};
    (5,0)*{\scriptstyle\circ}="5s";
    (15,0)*{\scriptstyle\cdots\cdots};
    (25,0)*{\scriptstyle\circ}="6s";
    (-45,10)*{\scriptstyle\circ}="1d";
    (-35,10)*{\scriptstyle\cdots\cdots};
    (-25,10)*{\scriptstyle\circ}="2d";
    (-15,10)*{\scriptstyle\circ}="3d";
    (-15,13)*{\scriptstyle i};
    (-5,10)*{\scriptstyle\circ}="4d";
    (-5,13)*{\scriptstyle i+1};
    (5,10)*{\scriptstyle\circ}="5d";
    (15,10)*{\scriptstyle\cdots\cdots};
    (25,10)*{\scriptstyle\circ}="6d";
    "1s"; "1d" **\dir{-};
    "2s"; "2d" **\dir{-};
    "3s"; "4s" **\crv{(-14,3.5) & (-6,3.5)};
    "3d"; "4d" **\crv{(-14,6.5) & (-6,6.5)};
    "5s"; "5d" **\dir{-};
    "6s"; "6d" **\dir{-};
    \endxy
    \]
 \caption{$e_i$}
	\label{Fig3}
\end{figure}
\noindent which can be considered as $e_i$ for $i=1, 2, \cdots, n-1$.
Note that we label the dots of both rows from left to right by $1, 2, \cdots, n$, and the identity 1
could be considered as the diagram in which the $i$-th dot in the top row
joins to $i$-th one in the bottom row for all $i=1, 2, \cdots, n$.

The product of  two arbitrary diagrams
is defined to be their concatenation with the rules in Figure \ref{Fig4} on removing circuits and decorations.
\begin{figure}[H]
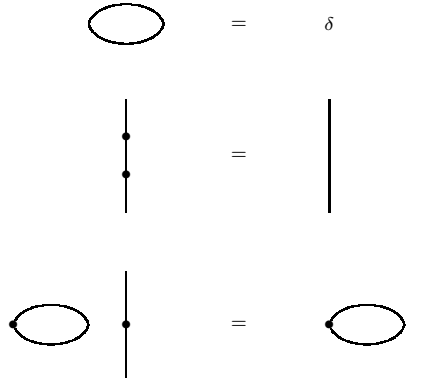

	\[
	\xy
    (-15,0)*{\scriptstyle}="1s";
    (-5,0)*{\scriptstyle}="2s";
    (5,0)*{\scriptstyle =};
    (17,0)*{\scriptstyle \delta};
    "1s"; "2s" **\crv{(-14,-3.5) & (-6,-3.5)};
    "1s"; "2s" **\crv{(-14,3.5) & (-6,3.5)};
    (-10,-10)*{\scriptstyle}="3s";
    (-10,-25)*{\scriptstyle}="4s";
    (-10,-15)*{\scriptstyle\bullet};
    (-10,-20)*{\scriptstyle\bullet};
    (5,-17.5)*{\scriptstyle =};
    (17,-10)*{\scriptstyle}="5s";
    (17,-25)*{\scriptstyle}="6s";
    "3s"; "4s" **\dir{-};
    "5s"; "6s" **\dir{-};
    (-25,-40)*{\scriptstyle\bullet}="7s";
    (-15,-40)*{\scriptstyle}="8s";
    "7s"; "8s" **\crv{(-24,-36.5) & (-16,-36.5)};
    "7s"; "8s" **\crv{(-24,-43.5) & (-16,-43.5)};
    (-10,-33)*{\scriptstyle}="9s";
    (-10,-47)*{\scriptstyle}="10s";
    "9s"; "10s" **\dir{-};
    (-10,-40)*{\scriptstyle\bullet};
    (5,-40)*{\scriptstyle =};
    (17,-40)*{\scriptstyle\bullet}="11s";
    (27,-40)*{\scriptstyle}="12s";
    "11s"; "12s" **\crv{(18,-36.5) & (26,-36.5)};
    "11s"; "12s" **\crv{(18,-43.5) & (26,-43.5)};
    (32,-33)*{\scriptstyle}="13s";
    (32,-47)*{\scriptstyle}="14s";
    "13s"; "14s" **\dir{-};
    \endxy
    \]
    \caption{Removing rule}
	\label{Fig4}
\end{figure}

Note that the first rule in Figure \ref{Fig4} means that a circuit can be replaced by $\delta$.
The second one means that arbitrary two decorations on a curve can be removed together.
The third one means that any decoration can be removed whenever it encounter a decorated circuit.

Consequently, if we ignore a nonzero scalar, then any word in the generators $e_{\bar{1}}, e_1, \cdots, e_{n-1}$
corresponds to a diagram, which will be called from now on a decorated Temperley-Lieb $n$-diagram.

To describe more details about the set of decorated Temperley-Lieb diagrams,
let us recall the definition of a Temperley-Lieb diagram, which
is a ``planar Brauer diagram" (we refer the reader to \cite{J3}).
A Temperley-Lieb diagram consists of two rows of $n$ dots in which each dot is joined to just one other dot with an arc,
and none of arcs intersect in the rectangle defined by the $2n$ dots. An arc
is called to be horizontal if it joins two dots in the same row and said to be vertical otherwise.

Green proved in \cite{G2} the following result, which essentially provides a diagram basis of ${\rm TLD}_n(\delta)$.
\begin{lem}\label{3.3}
There are two types of decorated Temperley-Lieb diagrams.
A diagram of the first type contains a Temperley-Lieb diagram (not the identity) with a decorated circuit
in addition. A diagram of the second type is a Temperley-Lieb diagram with decorations on arcs satisfying the following conditions:
\begin{enumerate}
\item[(1)] contains no circuit;
\item[(2)] there is at most one decoration on each arc;
\item[(3)] the total number of decorations is even;
\item[(4)] any decorated arc have to be exposed to the left boundary of the rectangle defined by the $2n$ dots.
\end{enumerate}
\end{lem}

The diagram in Figure \ref{Fig7} is {\bf not} a decorated Temperley-Lieb diagram,
because it does not satisfy the condition (4) in Lemma \ref{3.3}.
	
\begin{figure}[H]
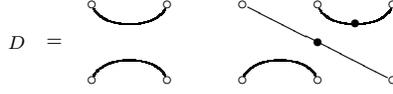

	\[
	\xy
    (-25,5)*{\scriptstyle D};
    (-20,5)*{\scriptstyle =};
    (-15,0)*{\scriptstyle\circ}="1s";
    (20,7.5)*{\scriptstyle\bullet};
    (-5,0)*{\scriptstyle\circ}="2s";
    (5,0)*{\scriptstyle\circ}="3s";
    (15,0)*{\scriptstyle\circ}="4s";
    (25,0)*{\scriptstyle\circ}="5s";
    (-15,10)*{\scriptstyle\circ}="1d";
    (15,5)*{\scriptstyle\bullet};
    (-5,10)*{\scriptstyle\circ}="2d";
    (5,10)*{\scriptstyle\circ}="3d";
    (15,10)*{\scriptstyle\circ}="4d";
    (25,10)*{\scriptstyle\circ}="5d";
    "1s"; "2s" **\crv{(-14,3.5) & (-6,3.5)};
    "1d"; "2d" **\crv{(-14,6.5) & (-6,6.5)};
    "5s"; "3d" **\dir{-};
    "3s"; "4s" **\crv{(6,3.5) & (14,3.5)};
    "4d"; "5d" **\crv{(16,6.5) & (24,6.5)};
    \endxy
    \]
 \caption{Counterexample}
	\label{Fig7}
\end{figure}

Let us illustrate the product of two decorated Temperley-Lieb $n$-diagrams.

\begin{example}\label{3.2}
Take $D_1$ and $D_2$ to be
\begin{figure}[H]
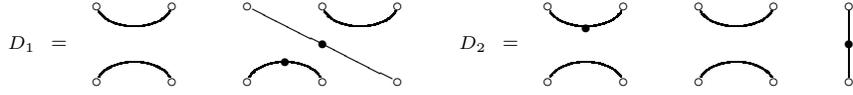

	\[
	\xy
    (-25,5)*{\scriptstyle D_1};
    (-20,5)*{\scriptstyle =};
    (-15,0)*{\scriptstyle\circ}="1s";
    (10,2.5)*{\scriptstyle\bullet};
    (-5,0)*{\scriptstyle\circ}="2s";
    (5,0)*{\scriptstyle\circ}="3s";
    (15,0)*{\scriptstyle\circ}="4s";
    (25,0)*{\scriptstyle\circ}="5s";
    (-15,10)*{\scriptstyle\circ}="1d";
    (15,5)*{\scriptstyle\bullet};
    (-5,10)*{\scriptstyle\circ}="2d";
    (5,10)*{\scriptstyle\circ}="3d";
    (15,10)*{\scriptstyle\circ}="4d";
    (25,10)*{\scriptstyle\circ}="5d";
    "1s"; "2s" **\crv{(-14,3.5) & (-6,3.5)};
    "1d"; "2d" **\crv{(-14,6.5) & (-6,6.5)};
    "5s"; "3d" **\dir{-};
    "3s"; "4s" **\crv{(6,3.5) & (14,3.5)};
    "4d"; "5d" **\crv{(16,6.5) & (24,6.5)};
    (35,5)*{\scriptstyle D_2};
    (40,5)*{\scriptstyle =};
    (45,0)*{\scriptstyle\circ}="6s";
    (50,7)*{\scriptstyle\bullet};
    (55,0)*{\scriptstyle\circ}="7s";
    (65,0)*{\scriptstyle\circ}="8s";
    (75,0)*{\scriptstyle\circ}="9s";
    (85,0)*{\scriptstyle\circ}="10s";
    (45,10)*{\scriptstyle\circ}="6d";
    (85,5)*{\scriptstyle\bullet};
    (55,10)*{\scriptstyle\circ}="7d";
    (65,10)*{\scriptstyle\circ}="8d";
    (75,10)*{\scriptstyle\circ}="9d";
    (85,10)*{\scriptstyle\circ}="10d";
    "6s"; "7s" **\crv{(46,3.5) & (54,3.5)};
    "6d"; "7d" **\crv{(46,6.5) & (54,6.5)};
    "8d"; "9d" **\crv{(66,6.5) & (74,6.5)};
    "8s"; "9s" **\crv{(66,3.5) & (74,3.5)};
    "10s"; "10d" **\dir{-};
    \endxy
    \]
    \caption{Decorated Temperley-Lieb diagrams}
	\label{Fig5}
\end{figure}
\noindent Note that both $D_1$ and $D_2$ in Figure \ref{Fig5} are diagrams of the second type.
Then by using the rule 2, rule 3 and rule 1 of Figure \ref{Fig4} in turn,
we get the diagram of the product, which is of the first type. We illustrate the process as follows.
\begin{figure}[H]
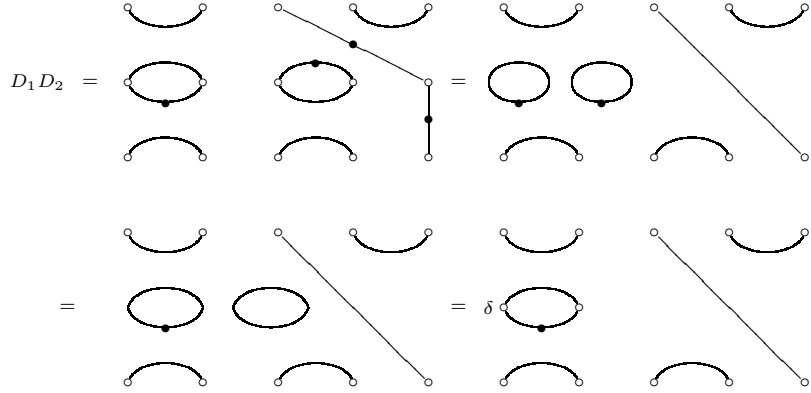

	\[
	\xy
    (-27,0)*{\scriptstyle D_1D_2};
    (-20,0)*{\scriptstyle =};
    (-15,0)*{\scriptstyle\circ}="1s";
    (10,2.5)*{\scriptstyle\bullet};
    (-5,0)*{\scriptstyle\circ}="2s";
    (5,0)*{\scriptstyle\circ}="3s";
    (15,0)*{\scriptstyle\circ}="4s";
    (25,0)*{\scriptstyle\circ}="5s";
    (-15,10)*{\scriptstyle\circ}="1d";
    (15,5)*{\scriptstyle\bullet};
    (-5,10)*{\scriptstyle\circ}="2d";
    (5,10)*{\scriptstyle\circ}="3d";
    (15,10)*{\scriptstyle\circ}="4d";
    (25,10)*{\scriptstyle\circ}="5d";
    (-15,-10)*{\scriptstyle\circ}="1a";
    (-5,-10)*{\scriptstyle\circ}="2a";
    (5,-10)*{\scriptstyle\circ}="3a";
    (15,-10)*{\scriptstyle\circ}="4a";
    (25,-10)*{\scriptstyle\circ}="5a";
    (25,-5)*{\scriptstyle\bullet};
    (-10,-2.8)*{\scriptstyle\bullet};
    "1s"; "2s" **\crv{(-14,3.5) & (-6,3.5)};
    "1s"; "2s" **\crv{(-14,-3.5) & (-6,-3.5)};
    "1d"; "2d" **\crv{(-14,6.5) & (-6,6.5)};
    "5s"; "3d" **\dir{-};
    "3s"; "4s" **\crv{(6,3.5) & (14,3.5)};
    "3s"; "4s" **\crv{(6,-3.5) & (14,-3.5)};
    "4d"; "5d" **\crv{(16,6.5) & (24,6.5)};
    "1a"; "2a" **\crv{(-14,-6.5) & (-6,-6.5)};
    "3a"; "4a" **\crv{(6,-6.5) & (14,-6.5)};
    "5s"; "5a" **\dir{-};
    (29,0)*{\scriptstyle =};
    (33,0)*{\scriptstyle}="1b";
    (41,0)*{\scriptstyle}="2b";
    (44,0)*{\scriptstyle}="3b";
    (52,0)*{\scriptstyle}="4b";
    (35,10)*{\scriptstyle\circ}="1f";
    (45,10)*{\scriptstyle\circ}="2f";
    (55,10)*{\scriptstyle\circ}="3f";
    (65,10)*{\scriptstyle\circ}="4f";
    (75,10)*{\scriptstyle\circ}="5f";
    (35,-10)*{\scriptstyle\circ}="1c";
    (45,-10)*{\scriptstyle\circ}="2c";
    (55,-10)*{\scriptstyle\circ}="3c";
    (65,-10)*{\scriptstyle\circ}="4c";
    (75,-10)*{\scriptstyle\circ}="5c";
    (37,-2.8)*{\scriptstyle\bullet};
    (48,-2.8)*{\scriptstyle\bullet};
    "1b"; "2b" **\crv{(33,3.5) & (41,3.5)};
    "1b"; "2b" **\crv{(33,-3.5) & (41,-3.5)};
    "3b"; "4b" **\crv{(44,3.5) & (52,3.5)};
    "3b"; "4b" **\crv{(44,-3.5) & (52,-3.5)};
    "1f"; "2f" **\crv{(36,6.5) & (44,6.5)};
    "4f"; "5f" **\crv{(66,6.5) & (74,6.5)};
    "1c"; "2c" **\crv{(36,-6.5) & (44,-6.5)};
    "3c"; "4c" **\crv{(56,-6.5) & (64,-6.5)};
    "3f"; "5c" **\dir{-};
    (-23,-30)*{\scriptstyle =};
    (-15,-30)*{\scriptstyle}="1b";
    (-5,-30)*{\scriptstyle}="2b";
    (-1,-30)*{\scriptstyle}="3b";
    (9,-30)*{\scriptstyle}="4b";
    (-15,-20)*{\scriptstyle\circ}="1f";
    (-5,-20)*{\scriptstyle\circ}="2f";
    (5,-20)*{\scriptstyle\circ}="3f";
    (15,-20)*{\scriptstyle\circ}="4f";
    (25,-20)*{\scriptstyle\circ}="5f";
    (-15,-40)*{\scriptstyle\circ}="1c";
    (-5,-40)*{\scriptstyle\circ}="2c";
    (5,-40)*{\scriptstyle\circ}="3c";
    (15,-40)*{\scriptstyle\circ}="4c";
    (25,-40)*{\scriptstyle\circ}="5c";
    (-10,-32.8)*{\scriptstyle\bullet};
    "1b"; "2b" **\crv{(-14,-26.5) & (-6,-26.5)};
    "1b"; "2b" **\crv{(-14,-33.5) & (-6,-33.5)};
    "3b"; "4b" **\crv{(0,-26.5) & (8,-26.5)};
    "3b"; "4b" **\crv{(0,-33.5) & (8,-33.5)};
    "1f"; "2f" **\crv{(-14,-23.5) & (-6,-23.5)};
    "4f"; "5f" **\crv{(16,-23.5) & (24,-23.5)};
    "1c"; "2c" **\crv{(-14,-36.5) & (-6,-36.5)};
    "3c"; "4c" **\crv{(6,-36.5) & (14,-36.5)};
    "3f"; "5c" **\dir{-};
    (29,-30)*{\scriptstyle =};
    (33,-30)*{\scriptstyle \delta};
    (35,-30)*{\scriptstyle\circ}="1b";
    (45,-30)*{\scriptstyle\circ}="2b";
    (35,-20)*{\scriptstyle\circ}="1f";
    (45,-20)*{\scriptstyle\circ}="2f";
    (55,-20)*{\scriptstyle\circ}="3f";
    (65,-20)*{\scriptstyle\circ}="4f";
    (75,-20)*{\scriptstyle\circ}="5f";
    (35,-40)*{\scriptstyle\circ}="1c";
    (45,-40)*{\scriptstyle\circ}="2c";
    (55,-40)*{\scriptstyle\circ}="3c";
    (65,-40)*{\scriptstyle\circ}="4c";
    (75,-40)*{\scriptstyle\circ}="5c";
    (40,-32.8)*{\scriptstyle\bullet};
    "1b"; "2b" **\crv{(36,-26.5) & (44,-26.5)};
    "1b"; "2b" **\crv{(36,-33.5) & (44,-33.5)};
    "1f"; "2f" **\crv{(36,-23.5) & (44,-23.5)};
    "4f"; "5f" **\crv{(66,-23.5) & (74,-23.5)};
    "1c"; "2c" **\crv{(36,-36.5) & (44,-36.5)};
    "3c"; "4c" **\crv{(56,-36.5) & (64,-36.5)};
    "3f"; "5c" **\dir{-};
    \endxy
    \]
    \caption{Product of diagrams}
	\label{Fig6}
\end{figure}
Note that in Figure \ref{Fig6} we can get the third diagram from the first one by using the third rule only.
\end{example}

\medskip

\section{Decorated parenthesis diagrams}

A Temperley-Lieb diagram can be split into a top part and a bottom part and each
part is in fact a parenthesis diagram. In this section, we generalize
the parenthesis diagrams to decorated ones and give some results on them which are needed in the following sections.
In particular, we define an order for the set of decorated parenthesis diagrams and study some mappings related to it.

\subsection{Definition of decorated parenthesis diagrams}
 We first recall the definition of a parenthesis diagram, which can be found in \cite{W2}.
\begin{dfn}\label{4.1}
A parenthesis diagram is an involution $\sigma$ of a totally ordered finite set such that
\begin{enumerate}
\item[(1)]\,there do not exist $i, j$ and $k$ with $i<j<k$, $\sigma(j)=j$ and $\sigma(i)=k$;
\item[(2)]\,there do not exist $i, j, k$ and $l$ with $i<j<k<l$ and $\sigma(i)=k$ and $\sigma(j)=l$.
\end{enumerate}
If the totally ordered set has $n$ elements and there are just $2p$ elements $i$ with $\sigma(i)\neq i$,
then $\sigma$ is an $(n, p)$-parenthesis diagram.
\end{dfn}    	
	
Given a parenthesis diagram $\sigma$, if $\sigma(i)=i$, then $i$ is called an isolated dot.
If $\sigma(i)=j$ with $i\neq j$, then $(i, j)$ is called a pair in $\sigma$.
We write a pair as $(i, j)$ implying $i<j$.
A decoration set $\mathbb{D}$ consists of pairs in $\sigma$ satisfying
\begin{enumerate}
\item[(1)]\, if $(i, j)\in \mathbb{D}$ and $\sigma(k)=k$, then $k>j$;
\item[(2)]\, if $(k, l)\in \mathbb{D}$, then any pair $(i, j)$ with $i<k<l<j$ does not in $\mathbb{D}$.
\end{enumerate}
Each element in $\mathbb{D}$ is called a decoration of $\sigma$.

\begin{dfn}\label{4.2}
A parenthesis diagram $\sigma$ with a decoration set $\mathbb{D}$ is called a
decorated parenthesis diagram and denoted by $(\sigma, \mathbb{D})$.
\end{dfn}

There are various ways to write a parenthesis diagram.
We will choose the way of the so-called ``cup (cap) diagram". Let us illustrate an example here.
\begin{example}\label{4.3}
Take an involution $\sigma=(1, 4)(2, 3)(5)$. Then the ``cup diagram" and ``cap diagram" of it are as follows.
\begin{figure}[H]
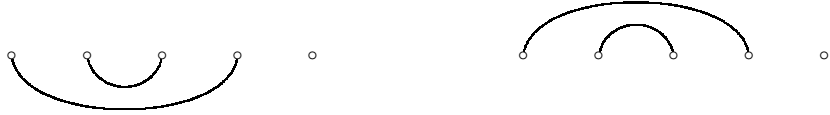

	\[
	\xy
    (-43,0)*{\scriptstyle\circ}="1s";
    (-33,0)*{\scriptstyle\circ}="2s";
    (-23,0)*{\scriptstyle\circ}="3s";
    (-13,0)*{\scriptstyle\circ}="4s";
    (-3,0)*{\scriptstyle\circ}="5s";
    "1s"; "4s" **\crv{(-42,-9.5) & (-14,-9.5)};
    "2s"; "3s" **\crv{(-32,-5.5) & (-24,-5.5)};
    (25,0)*{\scriptstyle\circ}="1d";
    (35,0)*{\scriptstyle\circ}="2d";
    (45,0)*{\scriptstyle\circ}="3d";
    (55,0)*{\scriptstyle\circ}="4d";
    (65,0)*{\scriptstyle\circ}="5d";
    "1d"; "4d" **\crv{(26,9.5) & (54,9.5)};
    "2d"; "3d" **\crv{(36,5.5) & (44,5.5)};
    \endxy
    \]
    \caption{Cup diagram and cap diagram}
	\label{Fig8}
\end{figure}
\end{example}

Clearly, a pair in a parenthesis diagram corresponds to an arc in its cup (cap) diagram. Then
we can write a decorated parenthesis diagram as a decorated cup (cap) diagram, in which
an arc is decorated by a dot if the corresponding pair is in the decoration set.
According to Definition \ref{4.2}, in a decorated cup diagram,
there is not any isolated dot on the left side of a decorated arc.
Let $\sigma$ be as above and the decoration set be (1, 4). Then the corresponding decorated cup diagram is
\begin{figure}[H]
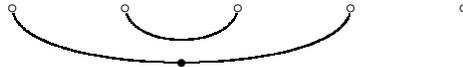

	\[
	\xy
    (-15,0)*{\scriptstyle\circ}="1s";
    (0,0)*{\scriptstyle\circ}="2s";
    (15,0)*{\scriptstyle\circ}="3s";
    (30,0)*{\scriptstyle\circ}="4s";
    (7.5,-7.3)*{\scriptstyle\bullet};
    (45,0)*{\scriptstyle\circ}="5s";
    "1s"; "4s" **\crv{(-14,-9.5) & (29,-9.5)};
    "2s"; "3s" **\crv{(1,-5.5) & (14,-5.5)};
    \endxy
    \]
    \caption{Decorated cup diagram}
	\label{Fig9}
\end{figure}
Denote the set of $(n, p)$-decorated  parenthesis diagrams by $M(n-2p)_n$.
We often omit the subscript $n$ if there is no danger of confusion.
Given a set $X$, the cardinal number of $X$ is denoted by $|X|$.
\begin{lem}\label{4.4}
Let $2p\leq n$. Then
$|M(n-2p)_n|=\left(\begin{array}{c} n \\ p \\ \end{array} \right)$.
\end{lem}

\begin{proof}
Denote the number of $(n, p)$-decorated parenthesis diagrams by $d(n, p)$.
We claim that these numbers are determined by the recurrence relation
$$
d(n, p)=\left\{
           \begin{array}{ll}
             d(n-1, p), & \hbox{if $p=0$;} \\
             d(n-1, p)+d(n-1, p-1), & \hbox{if $0<2p<n$;} \\
             2d(n-1, p-1), & \hbox{if $n=2p$.}
           \end{array}
         \right.
$$
In fact, if $p=0$, then $d(n, p)=d(n-1, p)$ is clear. Assume that $0<2p<n$. For an
$(n, p)$-decorated parenthesis diagram $\sigma$ with decoration set $\mathbb{D}$,
which denoted by $(\sigma, \mathbb{D})$, if $n$ is an isolated dot,
then removing $n$ yields an $(n-1, p)$-decorated parenthesis diagram with decoration set $\mathbb{D}$;
if $\sigma(i)=n$ with $i\neq n$ , then there exists an isolated dot $j$ with $j<i$ because $0<2p<n$.
As a result, $(i,n)\notin\mathbb{D}$. Then removing $n$ and letting $i$ to be an isolated dot gives
an $(n-1, p-1)$-decorated parenthesis diagram with decoration set $\mathbb{D}$. This process is invertable clearly.

For the last case, suppose that $\sigma(i)=n$ and $(i, n)\notin \mathbb{D}$, where $\mathbb{D}$ is a set of pairs of $\sigma$. Then
$(\sigma, \mathbb{D})$ and $(\sigma, \mathbb{D}\cup (i, n))$ are two deferent decorated parenthesis diagrams.
However, the $(n-1, p-1)$-decorated parenthesis diagrams obtained from that two diagrams by removing $n$ and leaving $i$ isolated are the same.
Consequently, $d(2p, p)=2d(2p-1, p-1)$. It is easy to check that
$\left(\begin{array}{c}n \\p \\\end{array}\right)$ also satisfy the recurrence relation and this completes the proof.
\end{proof}

Let us define an order on the set of $(n, p)$-decorated parenthesis diagrams.
The order plays a key role in studying the Gram matrices of cell modules.

\begin{dfn}\label{4.5}
Given an $(n, p)$-decorated parenthesis diagram with $t$ decorations, we define a sequence of $2p$ integers
$(i_1, i_2, \cdots, i_p, i_{p+1}, i_{p+2}, \cdots, i_{2p})$ associated with it, where $i_1>i_2>\cdots>i_p$ are
the second numbers of the $p$-pairs respectively, and where $i_{p+1}<\cdots < i_{p+t}$ are the second numbers of
the $t$-decorations respectively and $i_{p+t+1}=\cdots=i_{2p}=\infty$.
Define a total order on the
set $M(n-2p)_n$ by taking the lexicographic order on the sequences given above.
\end{dfn}

\begin{example}\label{4.6}
In the order of Definition \ref{4.5}, the $(5, 2)$-decorated parenthesis diagrams are arranged in Figure \ref{Fig10}.
\begin{figure}[H]
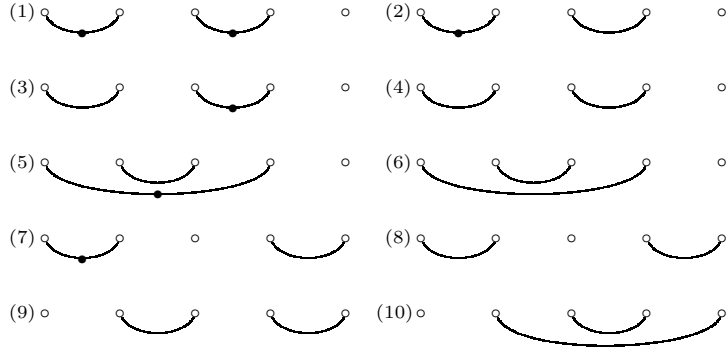

	\[
	\xy
    (-18,0)*{\scriptstyle (1)};
    (-15,0)*{\scriptstyle\circ}="1s";
    (-10,-2.8)*{\scriptstyle\bullet};
    (-5,0)*{\scriptstyle\circ}="2s";
    (5,0)*{\scriptstyle\circ}="3s";
    (10,-2.8)*{\scriptstyle\bullet};
    (15,0)*{\scriptstyle\circ}="4s";
    (25,0)*{\scriptstyle\circ}="5s";
    "1s"; "2s" **\crv{(-14,-3.5) & (-6,-3.5)};
    "3s"; "4s" **\crv{(6,-3.5) & (14,-3.5)};
    (32,0)*{\scriptstyle (2)};
    (35,0)*{\scriptstyle\circ}="6s";
    (40,-2.8)*{\scriptstyle\bullet};
    (45,0)*{\scriptstyle\circ}="7s";
    (55,0)*{\scriptstyle\circ}="8s";
    (65,0)*{\scriptstyle\circ}="9s";
    (75,0)*{\scriptstyle\circ}="10s";
    "6s"; "7s" **\crv{(36,-3.5) & (44,-3.5)};
    "8s"; "9s" **\crv{(56,-3.5) & (64,-3.5)};
    (-18,-10)*{\scriptstyle (3)};
    (-15,-10)*{\scriptstyle\circ}="1d";
    (10,-12.8)*{\scriptstyle\bullet};
    (-5,-10)*{\scriptstyle\circ}="2d";
    (5,-10)*{\scriptstyle\circ}="3d";
    (15,-10)*{\scriptstyle\circ}="4d";
    (25,-10)*{\scriptstyle\circ}="5d";
    "1d"; "2d" **\crv{(-14,-13.5) & (-6,-13.5)};
    "3d"; "4d" **\crv{(6,-13.5) & (14,-13.5)};
    (32,-10)*{\scriptstyle (4)};
    (35,-10)*{\scriptstyle\circ}="6d";
    (45,-10)*{\scriptstyle\circ}="7d";
    (55,-10)*{\scriptstyle\circ}="8d";
    (65,-10)*{\scriptstyle\circ}="9d";
    (75,-10)*{\scriptstyle\circ}="10d";
    "6d"; "7d" **\crv{(36,-13.5) & (44,-13.5)};
    "8d"; "9d" **\crv{(56,-13.5) & (64,-13.5)};
    (-18,-20)*{\scriptstyle (5)};
    (-15,-20)*{\scriptstyle\circ}="1f";
    (-5,-20)*{\scriptstyle\circ}="2f";
    (0,-24.2)*{\scriptstyle\bullet};
    (5,-20)*{\scriptstyle\circ}="3f";
    (15,-20)*{\scriptstyle\circ}="4f";
    (25,-20)*{\scriptstyle\circ}="5f";
    "1f"; "4f" **\crv{(-14,-25.5) & (14,-25.5)};
    "2f"; "3f" **\crv{(-4,-23.5) & (4,-23.5)};
    (32,-20)*{\scriptstyle (6)};
    (35,-20)*{\scriptstyle\circ}="6f";
    (45,-20)*{\scriptstyle\circ}="7f";
    (55,-20)*{\scriptstyle\circ}="8f";
    (65,-20)*{\scriptstyle\circ}="9f";
    (75,-20)*{\scriptstyle\circ}="10f";
    "6f"; "9f" **\crv{(36,-25.5) & (64,-25.5)};
    "7f"; "8f" **\crv{(46,-23.5) & (54,-23.5)};
    (-18,-30)*{\scriptstyle (7)};
    (-15,-30)*{\scriptstyle\circ}="1g";
    (-10,-32.8)*{\scriptstyle\bullet};
    (-5,-30)*{\scriptstyle\circ}="2g";
    (5,-30)*{\scriptstyle\circ}="3g";
    (15,-30)*{\scriptstyle\circ}="4g";
    (25,-30)*{\scriptstyle\circ}="5g";
    "1g"; "2g" **\crv{(-14,-33.5) & (-6,-33.5)};
    "4g"; "5g" **\crv{(16,-33.5) & (24,-33.5)};
    (32,-30)*{\scriptstyle (8)};
    (35,-30)*{\scriptstyle\circ}="6g";
    (45,-30)*{\scriptstyle\circ}="7g";
    (55,-30)*{\scriptstyle\circ}="8g";
    (65,-30)*{\scriptstyle\circ}="9g";
    (75,-30)*{\scriptstyle\circ}="10g";
    "6g"; "7g" **\crv{(36,-33.5) & (44,-33.5)};
    "9g"; "10g" **\crv{(66,-33.5) & (74,-33.5)};
    (-18,-40)*{\scriptstyle (9)};
    (-15,-40)*{\scriptstyle\circ}="1h";
    (-5,-40)*{\scriptstyle\circ}="2h";
    (5,-40)*{\scriptstyle\circ}="3h";
    (15,-40)*{\scriptstyle\circ}="4h";
    (25,-40)*{\scriptstyle\circ}="5h";
    "2h"; "3h" **\crv{(-4,-43.5) & (4,-43.5)};
    "4h"; "5h" **\crv{(16,-43.5) & (24,-43.5)};
    (31.5,-40)*{\scriptstyle (10)};
    (35,-40)*{\scriptstyle\circ}="6h";
    (45,-40)*{\scriptstyle\circ}="7h";
    (55,-40)*{\scriptstyle\circ}="8h";
    (65,-40)*{\scriptstyle\circ}="9h";
    (75,-40)*{\scriptstyle\circ}="10h";
    "7h"; "10h" **\crv{(46,-46) & (74,-45.5)};
    "8h"; "9h" **\crv{(56,-43.5) & (64,-43.5)};
    \endxy
    \]
    \caption{The order}
	\label{Fig10}
\end{figure}
\end{example}

\subsection{Some maps preserving the order}
Denote the set of decorated $(2p, p)$-parenthesis diagrams with $|\mathbb{D}|$ being even
by $M^+(0)$, and similarly denote by $M^-(0)$ the diagrams having odd decorations.
We construct a map $\alpha: M^+(0)\rightarrow M^-(0)$ as follows.
For $(\sigma, \mathbb{D})\in M^+(0)$ with $\sigma(i)=n$, define
$$
\alpha(\sigma, \mathbb{D})=\begin{cases} (\sigma, \mathbb{D}-(i, n)),&\text{if $(i, n)\in \mathbb{D}$;}\\
(\sigma, \mathbb{D}\cup(i, n)),\, &\text{if $(i, n)\notin \mathbb{D}$.}\end{cases}
$$

\begin{example}
The following are $(4, 2)$-decorated parenthesis diagrams arranged in the order of Definition \ref{4.5}.
\begin{figure}[H]
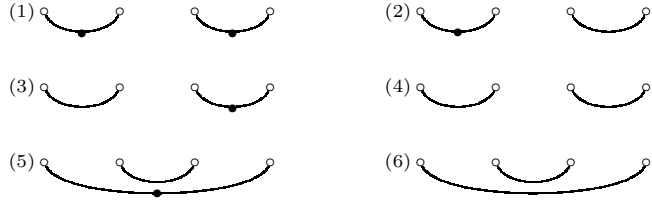

	\[
	\xy
    (-18,0)*{\scriptstyle (1)};
    (-15,0)*{\scriptstyle\circ}="1s";
    (-10,-2.8)*{\scriptstyle\bullet};
    (-5,0)*{\scriptstyle\circ}="2s";
    (5,0)*{\scriptstyle\circ}="3s";
    (10,-2.8)*{\scriptstyle\bullet};
    (15,0)*{\scriptstyle\circ}="4s";
    "1s"; "2s" **\crv{(-14,-3.5) & (-6,-3.5)};
    "3s"; "4s" **\crv{(6,-3.5) & (14,-3.5)};
    (32,0)*{\scriptstyle (2)};
    (35,0)*{\scriptstyle\circ}="6s";
    (40,-2.8)*{\scriptstyle\bullet};
    (45,0)*{\scriptstyle\circ}="7s";
    (55,0)*{\scriptstyle\circ}="8s";
    (65,0)*{\scriptstyle\circ}="9s";
    "6s"; "7s" **\crv{(36,-3.5) & (44,-3.5)};
    "8s"; "9s" **\crv{(56,-3.5) & (64,-3.5)};
    (-18,-10)*{\scriptstyle (3)};
    (-15,-10)*{\scriptstyle\circ}="1d";
    (10,-12.8)*{\scriptstyle\bullet};
    (-5,-10)*{\scriptstyle\circ}="2d";
    (5,-10)*{\scriptstyle\circ}="3d";
    (15,-10)*{\scriptstyle\circ}="4d";
    "1d"; "2d" **\crv{(-14,-13.5) & (-6,-13.5)};
    "3d"; "4d" **\crv{(6,-13.5) & (14,-13.5)};
    (32,-10)*{\scriptstyle (4)};
    (35,-10)*{\scriptstyle\circ}="6d";
    (45,-10)*{\scriptstyle\circ}="7d";
    (55,-10)*{\scriptstyle\circ}="8d";
    (65,-10)*{\scriptstyle\circ}="9d";
    "6d"; "7d" **\crv{(36,-13.5) & (44,-13.5)};
    "8d"; "9d" **\crv{(56,-13.5) & (64,-13.5)};
    (-18,-20)*{\scriptstyle (5)};
    (-15,-20)*{\scriptstyle\circ}="1f";
    (-5,-20)*{\scriptstyle\circ}="2f";
    (0,-24.2)*{\scriptstyle\bullet};
    (5,-20)*{\scriptstyle\circ}="3f";
    (15,-20)*{\scriptstyle\circ}="4f";
    "1f"; "4f" **\crv{(-14,-25.5) & (14,-25.5)};
    "2f"; "3f" **\crv{(-4,-23.5) & (4,-23.5)};
    (32,-20)*{\scriptstyle (6)};
    (35,-20)*{\scriptstyle\circ}="6f";
    (45,-20)*{\scriptstyle\circ}="7f";
    (55,-20)*{\scriptstyle\circ}="8f";
    (65,-20)*{\scriptstyle\circ}="9f";
    "6f"; "9f" **\crv{(36,-25.5) & (64,-25.5)};
    "7f"; "8f" **\crv{(46,-23.5) & (54,-23.5)};
    \endxy
    \]
    \caption{(4,2)-diagrams}
	\label{(4,2)}
\end{figure}
The diagrams in $M^+(0)$ are (1), (4), (6) and they are mapped by $\alpha$ to (2), (3), (5), respectively.
Clearly, in this example $\alpha$ preserves the order in Definition \ref{4.5}.
\end{example}

For convenience, we always write $\alpha(\sigma, \mathbb{D})=(\alpha(\sigma), \alpha(\mathbb{D}))$
for arbitrary maps and $(n, p)$-decorated parenthesis diagrams.
It is easy to check that $\alpha$ is a bijection and we omit the details.

\begin{lem}\label{4.7}
$|M^+(0)|=|M^-(0)|$.
\end{lem}

The following lemma says that the map $\alpha$ is in fact an isomorphism between the
two totally ordered sets $M^+(0)$ and $M^-(0)$, where the order is that given in Definition \ref{4.5}.

\begin{lem}\label{4.8}
Let $(\sigma_1, \mathbb{D}_1), (\sigma_2, \mathbb{D}_2)\in M^+(0)$ with
$(\sigma_1, \mathbb{D}_1)<(\sigma_2, \mathbb{D}_2)$. Then
$$\alpha(\sigma_1, \mathbb{D}_1)<\alpha(\sigma_2, \mathbb{D}_2).$$
\end{lem}

\begin{proof}
Assume that the associated sequences of
$(\sigma_1, \mathbb{D}_1)$ and $(\sigma_2, \mathbb{D}_2)$ are $$(i_1, i_2, \cdots, i_p, i_{p+1}, i_{p+2}, \cdots, i_{2p})$$ and
$$(j_1, j_2, \cdots, j_p, j_{p+1}, j_{p+2}, \cdots, j_{2p}),$$ respectively.
Since $(\sigma_1, \mathbb{D}_1)<(\sigma_2, \mathbb{D}_2)$, there exists some $s$ such that $i_s<j_s$
and $i_k=j_k$ for $k=1, \cdots, s-1$.

If $s< p$,
then $\alpha(\sigma_1, \mathbb{D}_1)<\alpha(\sigma_2, \mathbb{D}_2)$ because $\alpha$
does not influence the first $p$ numbers of the associated sequence.

The case $s=p$ is impossible, because $i_p=n=j_p$.

If $s>p$, then $\sigma_1=\sigma_2$. We split the proof into the following three cases.

{\it Case 1.}\,\, $j_{s}< n$. As a result, $i_{s}< n$ too.
We have from the definition of $\alpha$
that the first $s$ numbers of the associated sequence of $(\sigma_1, \mathbb{D}_1)$,
and that of $(\sigma_2, \mathbb{D}_2)$ are not changed by $\alpha$.
This implies that $\alpha(\sigma_1, \mathbb{D}_1)<\alpha(\sigma_2, \mathbb{D}_2)$.

{\it Case 2.}\,\, $j_{s}=n$. By the definition of $\alpha$, the $s$-th number
of the associated sequence of $\alpha(\sigma_2, \mathbb{D}_2)$ is $\infty$ and the others are equal to that
of $(\sigma_2, \mathbb{D}_2)$. Note that $i_s<j_s$. Then $\alpha$ does not influence the first $s$
numbers of the associated sequence of $(\sigma_1, \mathbb{D}_1)$. Clearly, $i_s<\infty$ and thus
$\alpha(\sigma_1, \mathbb{D}_1)<\alpha(\sigma_2, \mathbb{D}_2)$.

{\it Case 3.}\,\, $j_{s}=\infty$. Then $i_s$ is a finite number. We assert that $i_s<n$.
In fact, if $i_s=n$, then there is no decoration but $(\sigma(n), n)$,
which is in $\mathbb{D}_1$ and not in $\mathbb{D}_2$. This implies that the
parities of $|\mathbb{D}_1|$ and $|\mathbb{D}_2|$ are not the same, which contradicts with the hypothetical condition.
Because $j_{s}=\infty$, we have $(\sigma(n), n)\notin \mathbb{D}_2$ if we employ Definition \ref{4.5}.
Consequently, the $s$-th number of the associated sequence of $\alpha(\sigma_2, \mathbb{D}_2)$
is $n$ and is more than $i_s$ still. However, the first $s$ numbers of the  associated sequence
of $\alpha(\sigma_1, \mathbb{D}_1)$ is not changed by $\alpha$ and thus the inequality holds in this case.

To sum up, either $s< p$ or $s>p$, the inequality
$\alpha(\sigma_1, \mathbb{D}_1)<\alpha(\sigma_2, \mathbb{D}_2)$ holds and the proof is completed.
\end{proof}

Let us consider other maps preserving the order in Definition \ref{4.5}.
Let $n=2p$. Then Lemma \ref{4.4} implies that $|M(1)_{n-1}|=|M^+(0)|$.
Define a map $$\beta^+: M^+(0)\rightarrow M(1)_{n-1}$$ as follows.
For $(\sigma, \mathbb{D})\in M^+(0)$ with $\sigma=(k_1, l_1)\cdots(k_{p-1}, l_{p-1})(k, n)$
 $$\beta^+(\sigma, \mathbb{D})=
((k_1, l_1)\cdots(k_{p-1}, l_{p-1})(k), \mathbb{D}-(k, n)).$$
Similarly, one can define a map $\beta^-$ from $M^-(0)$ to $M(1)_{n-1}$.

\begin{example}
Recall the diagrams in Figure \ref{(4,2)}. Then
\begin{figure}[H]
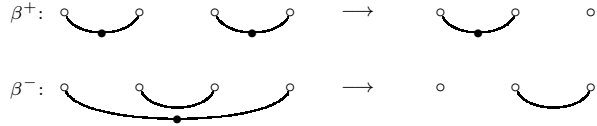

	\[
	\xy
    (-20,0)*{\scriptstyle \beta^+:};
    (-15,0)*{\scriptstyle\circ}="1s";
    (-10,-2.8)*{\scriptstyle\bullet};
    (-5,0)*{\scriptstyle\circ}="2s";
    (5,0)*{\scriptstyle\circ}="3s";
    (10,-2.8)*{\scriptstyle\bullet};
    (15,0)*{\scriptstyle\circ}="4s";
    "1s"; "2s" **\crv{(-14,-3.5) & (-6,-3.5)};
    "3s"; "4s" **\crv{(6,-3.5) & (14,-3.5)};
    (24,0)*{\scriptstyle \longrightarrow};
    (35,0)*{\scriptstyle\circ}="6s";
    (40,-2.8)*{\scriptstyle\bullet};
    (45,0)*{\scriptstyle\circ}="7s";
    (55,0)*{\scriptstyle\circ}="8s";
    "6s"; "7s" **\crv{(36,-3.5) & (44,-3.5)};
    (-20,-10)*{\scriptstyle \beta^-:};
    (-15,-10)*{\scriptstyle\circ}="1f";
    (-5,-10)*{\scriptstyle\circ}="2f";
    (0,-14.2)*{\scriptstyle\bullet};
    (5,-10)*{\scriptstyle\circ}="3f";
    (15,-10)*{\scriptstyle\circ}="4f";
    "1f"; "4f" **\crv{(-14,-15.5) & (14,-15.5)};
    "2f"; "3f" **\crv{(-4,-13.5) & (4,-13.5)};
    (24,-10)*{\scriptstyle \longrightarrow};
    (35,-10)*{\scriptstyle\circ}="6f";
    (45,-10)*{\scriptstyle\circ}="7f";
    (55,-10)*{\scriptstyle\circ}="8f";
    "7f"; "8f" **\crv{(46,-13.5) & (54,-13.5)};
    \endxy
    \]
    \caption{$\beta^+$ and $\beta^-$}
\end{figure}
\end{example}

Note that $\beta^+$ and $\beta^-$ are both bijections, and we omit the checking process.
Using the methods in Lemma \ref{4.8}, we obtain the following lemma.
Here we omit the proof and leave it to the reader as an exercise.
\begin{lem}\label{4.9}
Let $(\sigma_1, \mathbb{D}_1), (\sigma_2, \mathbb{D}_2)\in M^+(0)$ with
$(\sigma_1, \mathbb{D}_1)<(\sigma_2, \mathbb{D}_2)$. Then
$$\beta^+(\sigma_1, \mathbb{D}_1)<\beta^+(\sigma_2, \mathbb{D}_2)\,\,\,\,
{\rm and}\,\,\,\, \beta^-(\sigma_1, \mathbb{D}_1)<\beta^-(\sigma_2, \mathbb{D}_2).$$
\end{lem}

\smallskip

Let $\lam=n-2p$ for $n>2p$ and write the set of $(n, p)$-decorated parenthesis diagrams
$(\sigma, \mathbb{D})$ with $\sigma(n)\neq n$ by $M(\lam)|_{\sigma(n)\neq n}$.
Then by Lemma \ref{4.4} $$|M(\lam)|_{\sigma(n)\neq n}|=|M(n-2p+1)_{n-1}|.$$
Let $(\sigma, \mathbb{D})\in M(\lam)|_{\sigma(n)\neq n}$ with $\sigma=(k_1, l_1)\cdots(k_{p-1}, l_{p-1})(k, n)(j_1)\cdots (j_{n-2p})$.
It is helpful to point out that $(k, n)\notin \mathbb{D}$.
In fact, since $n>2p$, there exists $i$ such that $\sigma(i)=i$ with $i<k$.
Then we have from Definition \ref{4.2} that $(k, n)$ is an arc without a decoration.
Define a bijection $$\gamma: M(\lam)|_{\sigma(n)\neq n}\rightarrow M(n-2p+1)_{n-1}$$ by
$$\gamma(\sigma, \mathbb{D})=((k_1, l_1)\cdots(k_{p-1}, l_{p-1})(k)(j_1)\cdots (j_{n-2p}),\, \mathbb{D}).$$

\begin{example}
By the definition of $M(\lam)|_{\sigma(n)\neq n}$, the diagrams in $M(1)|_{\sigma(5)\neq 5}$ are Figure \ref{Fig10} (7)-(10).
\begin{figure}[H]
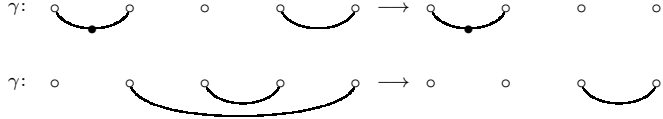

	\[
	\xy
    (-20,0)*{\scriptstyle \gamma:};
    (-15,0)*{\scriptstyle\circ}="1g";
    (-10,-2.8)*{\scriptstyle\bullet};
    (-5,0)*{\scriptstyle\circ}="2g";
    (5,0)*{\scriptstyle\circ}="3g";
    (15,0)*{\scriptstyle\circ}="4g";
    (25,0)*{\scriptstyle\circ}="5g";
    "1g"; "2g" **\crv{(-14,-3.5) & (-6,-3.5)};
    "4g"; "5g" **\crv{(16,-3.5) & (24,-3.5)};
    (30,0)*{\scriptstyle \longrightarrow};
    (35,0)*{\scriptstyle\circ}="6g";
    (40,-2.8)*{\scriptstyle\bullet};
    (45,0)*{\scriptstyle\circ}="7g";
    (55,0)*{\scriptstyle\circ}="8g";
    (65,0)*{\scriptstyle\circ}="9g";
    "6g"; "7g" **\crv{(36,-3.5) & (44,-3.5)};
    (-20,-10)*{\scriptstyle \gamma:};
    (-15,-10)*{\scriptstyle\circ}="6h";
    (-5,-10)*{\scriptstyle\circ}="7h";
    (5,-10)*{\scriptstyle\circ}="8h";
    (15,-10)*{\scriptstyle\circ}="9h";
    (25,-10)*{\scriptstyle\circ}="10h";
    "7h"; "10h" **\crv{(-4,-16) & (24,-15.5)};
    "8h"; "9h" **\crv{(6,-13.5) & (14,-13.5)};
    (30,-10)*{\scriptstyle \longrightarrow};
    (35,-10)*{\scriptstyle\circ}="6g";
    (45,-10)*{\scriptstyle\circ}="7g";
    (55,-10)*{\scriptstyle\circ}="8g";
    (65,-10)*{\scriptstyle\circ}="9g";
    "8g"; "9g" **\crv{(56,-13.5) & (64,-13.5)};
    \endxy
    \]
    \caption{The map $\gamma$}
\end{figure}
\end{example}

The following lemma is obvious.
\begin{lem}\label{4.10}
Let $(\sigma_1, \mathbb{D}_1), (\sigma_2, \mathbb{D}_2)\in M(\lam)|_{\sigma(n)\neq n}$ with
$(\sigma_1, \mathbb{D}_1)<(\sigma_2, \mathbb{D}_2)$. Then
$$\gamma(\sigma_1, \mathbb{D}_1)<\gamma(\sigma_2, \mathbb{D}_2).$$
\end{lem}

\medskip

\section{Branching rule for cell modules}
In this section, we first give the cellular structure of ${\rm TLD}_n(\delta)$
by $(n, p)$-decorated parenthesis diagrams, and then study the branching rule for the cell modules of ${\rm TLD}_n(\delta)$.
In particular, some notations will be fixed for later use.

\subsection{ Cellular structure of ${\rm TLD}_n(\delta)$}
Green \cite{G2} gave the diagrammatic realization of ${\rm TLD}_n(\delta)$ and
pointed out in Section 7 that the diagram calculi can be used to describe
the cellular structure. Unfortunately, Green did not write out all of the details.
As far as we know, there is not a precise reference which defines a cellular basis
for ${\rm TLD}_n(\delta)$ by diagrams. This diagrammatical cellular basis is very important for our aim.
Let us write it out here which can be considered as a supplement of Green's work \cite[Section 7]{G2}.

Let ${\rm TL}_n(\delta)$ be a Temperley-Lieb algebra of type A. The cellular structure of
${\rm TL}_n(\delta)$ is simple and well-known. Let us recall it and fix some notations which are needed in the sequel.
To make a difference with type D, we will change the font of letters.
The poset $\Lambda$ is $\{t\in\{1, 2, \cdots, n\}\mid n-t\in 2\mathbb{Z}\}$.
For $\lam\in\Lambda$, $\mathcal{M}(\lam)$ is the set of $(n, p)$-parenthesis diagrams. If $\mathcal{U,V}\in \mathcal{M}(\lam)$,
then $\mathcal{C}_{\mathcal{U,V}}^{\lam}$ is the diagram with $\mathcal{U}$ on the top and $\mathcal{V}$ on the bottom,
the isolated dots of $\mathcal{U}$ and $\mathcal{V}$ being joined in the unique way such that the diagram being planar.
For $\lam\in\Lambda$, the cell module will be denoted by $\mathcal{S}(n, p)$ and the Gram matrix by $\mathcal{G}(n, p)$.

\smallskip

Now let us construct a cellular basis for ${\rm TLD}_n(\delta)$. According to Definition \ref{2.1}
we first need to determine the poset $\Lambda$ and the associated set $M(\lam)$ for each $\lam\in\Lambda$.
Define $$\Lambda=\left\{
                \begin{array}{ll}
                  \{n> n-2> \cdots> 4> 2> 0^+> 0^-> \dot{n-2}> \cdots> \dot{2}> \dot{0} \}, & \hbox{if $n$ is even;} \\
                  \{n> n-2> \cdots >3 >1 > \dot{n-2} > \cdots> \dot{3}> \dot{1} \}, & \hbox{if $n$ is odd.}
                \end{array}
              \right.$$
If $\lam\in \{1, 2, \cdots, n\}$, then
$M(\lam)$ is the set of all $(n, \frac{n-\lam}{2})$-decorated parenthesis diagrams.
If $\lam=0^+$, then $M(\lam)$ is the set of $(n, \frac{n}{2})$-decorated parenthesis
diagrams with the number of decorations being even.
If $\lam=0^-$, then $M(\lam)$ consists of those with odd decorations.
If $\lam=\dot{k}$, where $k\in\{0, 1, 2, \cdots, n-2\}$,
then $M(\lam)$ is the set of $(n, \frac{n-k}{2})$ undecorated parenthesis diagrams.

Given a pair of ordered elements $S, T$ in $M(\lam)$, $C_{S, T}^{\lam}$ is a
decorated Temperley-Lieb diagram constructed as follows. If $\lam\neq \dot{k}, \,\,
k=0, 1, \cdots, n-2$, then put the decorated cup diagram of $S$
on the top  and put the decorated cap diagram of $T$ on the bottom. Then join the
isolated dots of $S$ with that of $T$ in the unique way. This process is completed if $s+t$ is even,
where $s$ and $t$ are the number of decorations of $S$ and $T$ respectively.
Otherwise, one need to add a decoration to the leftmost vertical arc.
Clearly, the diagram constructed is of the second type and all diagrams of the second type can be obtained by this way. If
$\lam=\dot{k}, \,\,k=0, 1, \cdots, n-2$, we first get a  Temperley-Lieb
diagram as in the case of type A and then add a decorated circuit. Clearly,
the diagrams obtained are the whole lot of the first type.

\begin{remarks}
(1) Given an arbitrary decorated diagram (the first type or the second type), we can obtain a unique pair of
decorated parenthesis diagrams by cutting the diagram horizontally along the middle
(ignoring the possible decorated circuit and the decorations of vertical arcs).

(2) Given a pair of decorated parenthesis diagrams, there may be two different decorated diagrams
because for $\lam, \mu\in \Lambda$, the intersection $M(\lam)\cap M(\mu)$ does not necessarily being empty.
Of course, another special case is that there being no decorated diagrams corresponding to them.
For example, two decorated parenthesis diagrams with different number of isolated dots.
\end{remarks}

Let us illustrate an example of forming a decorated Temperley-Lieb
diagram of the second type by two decorated parenthesis diagrams.

\begin{example}\label{5.1}
Let $n=7$, $\lam=3$, $\sigma=(1, 2)(3)(4, 5)(6)(7)$ with a decoration $(1, 2)$
and $\tau=(1)(2)(3, 6)(4, 5)(7)$ with no decoration. Then the corresponding
decorated Temperley-Lieb diagram is as follows.
\begin{figure}[H]
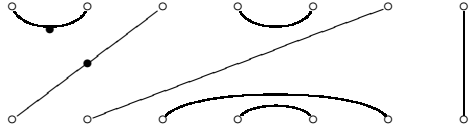

	\[
	\xy
    (-15,0)*{\scriptstyle\circ}="1s";
    (-5,0)*{\scriptstyle\circ}="2s";
    (5,0)*{\scriptstyle\circ}="3s";
    (15,0)*{\scriptstyle\circ}="4s";
    (25,0)*{\scriptstyle\circ}="5s";
    (35,0)*{\scriptstyle\circ}="6s";
    (45,0)*{\scriptstyle\circ}="7s";
    (-15,15)*{\scriptstyle\circ}="1d";
    (-5,15)*{\scriptstyle\circ}="2d";
    (-10,12)*{\scriptstyle\bullet};
    (-5,7.5)*{\scriptstyle\bullet};
    (5,15)*{\scriptstyle\circ}="3d";
    (15,15)*{\scriptstyle\circ}="4d";
    (25,15)*{\scriptstyle\circ}="5d";
    (35,15)*{\scriptstyle\circ}="6d";
    (45,15)*{\scriptstyle\circ}="7d";
    "1d"; "2d" **\crv{(-14,11.5) & (-6,11.5)};
    "4d"; "5d" **\crv{(16,11.5) & (24,11.5)};
    "3s"; "6s" **\crv{(7,4.5) & (33,4.5)};
    "4s"; "5s" **\crv{(16,2.5) & (24,2.5)};
    "1s"; "3d" **\dir{-};
    "2s"; "6d" **\dir{-};
    "7s"; "7d" **\dir{-};
    \endxy
    \]
    \caption{Construct cellular basis}
	\label{Fig11}
\end{figure}
\end{example}

To explain that the basis $\{C_{S, T}^{\lam}\}$ defined above is cellular, we need to check
(C2) and (C3) of Definition \ref{2.1}. Clearly, reversing $S$ and $T$ defines an anti-isomorphism of
${\rm TLD}_n(\delta)$ which satisfies (C2). Furthermore, it is easy to check that (C3) holds by
rules of the product of two diagrams and we omit the details here.

We need to fix some more notations that are needed in the sequel. Given the cellular basis defined above,
if $\lam=1, 2, \cdots, n$, we will denote the cell modules $W(\lam)$ by $S(n, p)$,
where $p=\frac{n-\lam}{2}$, and the corresponding Gram matrices are denoted by $G(n, p)$.
If $\lam=\dot{1}, \dot{2}, \cdots, \dot{n-2}$, we denote the cell modules
by $\dot{S}(n, p)$ and the Gram matrices $\dot{G}(n, p)$.
When $n=2p$, there are two cell modules. One has a basis includes all the decorated
parenthesis diagrams with even decorations, and this cell module is denoted
by $S^+(n, p)$ with the Gram matrix writing by $G^+(n, p)$. On the contrary,
if the number of decorations in each decorated parenthesis diagram is odd,
the cell module will be denoted by $S^-(n, p)$, and the Gram matrix $G^-(n, p)$.

\subsection{Branching rule for cell modules}

Let us consider the algebra ${\rm TLD}_{n-1}(\delta)$ as a subalgebra of ${\rm TLD}_{n}(\delta)$
canonically by adding a vertical arc to the right side of each decorated Temperley-Lieb
diagram in ${\rm TLD}_{n-1}(\delta)$. Along this inclusion, every cell module $S(n, p)$ of ${\rm TLD}_{n}(\delta)$
is also a ${\rm TLD}_{n-1}(\delta)$-module, which is denoted by $S(n, p)\downarrow$.

\begin{lem}\label{5.2}
Let $n$ and $p$ be positive integers with $n\geq 4$. We have
\begin{enumerate}
\item[(1)]\,If $n>2p+1$, then there is an exact sequence
$$0\rightarrow S(n-1, p)\rightarrow S(n, p)\downarrow\rightarrow S(n-1, p-1)\rightarrow 0.$$
\item[(2)]\,If $n=2p+1$, then there is an exact sequence
$$0\rightarrow S^+(2p, p)\oplus S^-(2p, p)\rightarrow S(n, p)\downarrow\rightarrow S(2p, p-1)\rightarrow 0.$$
\item[(3)]\,If $n=2p$, then $S^+(2p, p)\downarrow \cong S(2p-1, p-1)\cong S^-(2p, p)\downarrow.$
\end{enumerate}
\end{lem}

\begin{proof}
(1)\,Let $(\sigma, \mathbb{D})$ be an $(n-1, p)$-decorated parenthesis diagram, and assume that
$\sigma=(i_1, j_1)\cdots(i_p, j_p)(k_1)\cdots(k_{n-2p-1})$. Define a map $\iota: S(n-1, p)\rightarrow S(n, p)\downarrow$
by $\iota(\sigma, \mathbb{D})=(\sigma', \mathbb{D})$, where $\sigma'=(i_1, j_1)\cdots(i_p, j_p)(k_1)\cdots(k_{n-2p-1})(n)$.
Clearly, $\iota$ is an injective map. Moreover, it can be checked that $\iota$ is a ${\rm TLD}_{n-1}(\delta)$-module homomorphism.

We now define a map $\varsigma: S(n, p)\downarrow\rightarrow S(n-1, p-1)$ by
$$
\varsigma(\sigma, \mathbb{D})=\begin{cases} 0,& \text{if $n$ is an isolated dot;}\\
\gamma(\sigma, \mathbb{D}),\, &\text{otherwise.}\end{cases}
$$
We have from the definition of $\gamma$ that $\varsigma$ is surjective, and the sequence is an exact sequence of free $R$-modules.
It remains to prove that $\varsigma$ is a ${\rm TLD}_{n-1}(\delta)$-module homomorphism.
Let $(\sigma, \mathbb{D})$ be an $(n, p)$-decorated parenthesis diagram.
If $n$ is an isolated dot, then it is isolated in $a(\sigma, \mathbb{D})$ for each decorated Temperley-Lieb $(n-1)$-diagram $a$. This implies
$\varsigma(a(\sigma, \mathbb{D}))=a\varsigma(\sigma, \mathbb{D})$, and consequently, $\varsigma$ is a ${\rm TLD}_{n-1}(\delta)$-module homomorphism.
Then we only need to consider the case of $\sigma(n)\neq n$. We prove it by illustrating the map $\varsigma$ as follows.
\begin{figure}[H]
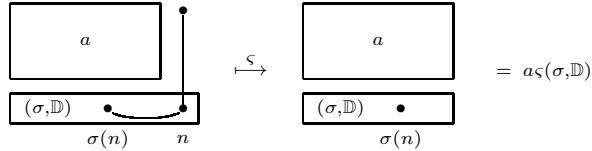

	\[
	\xy
    (-20,0)*{\scriptstyle}="1s";
    (0,0)*{\scriptstyle}="2s";
    (-20,10)*{\scriptstyle}="1d";
    (0,10)*{\scriptstyle}="2d";
    (-10,5)*{\scriptstyle a};
    "1s"; "1d" **\dir{-};
    "1s"; "2s" **\dir{-};
    "1d"; "2d" **\dir{-};
    "2s"; "2d" **\dir{-};
    (-20,-2)*{\scriptstyle}="1f";
    (5,-2)*{\scriptstyle}="2f";
    (-20,-6)*{\scriptstyle}="1g";
    (5,-6)*{\scriptstyle}="2g";
    (-15,-4)*{\scriptstyle (\sigma, \mathbb{D})};
    (3,-4)*{\scriptstyle\bullet}="3f";
    (-7,-4)*{\scriptstyle\bullet}="4f";
    (3,9)*{\scriptstyle\bullet}="3g";
    (3,-8)*{\scriptstyle n};
    (-7,-8)*{\scriptstyle \sigma(n)};
    "1f"; "1g" **\dir{-};
    "1f"; "2f" **\dir{-};
    "1g"; "2g" **\dir{-};
    "2f"; "2g" **\dir{-};
    "3f"; "3g" **\dir{-};
    "4f"; "3f" **\crv{(-6.5,-5.7) & (2.5,-5.7)};
    (12,1)*{\scriptstyle \longmapsto};
    (12,2.5)*{\scriptstyle \varsigma};
    \endxy\,\,\,\,\,\,\,\,
    \xy
    (-20,0)*{\scriptstyle}="1s";
    (0,0)*{\scriptstyle}="2s";
    (-20,10)*{\scriptstyle}="1d";
    (0,10)*{\scriptstyle}="2d";
    (-10,5)*{\scriptstyle a};
    "1s"; "1d" **\dir{-};
    "1s"; "2s" **\dir{-};
    "1d"; "2d" **\dir{-};
    "2s"; "2d" **\dir{-};
    (-20,-2)*{\scriptstyle}="1f";
    (0,-2)*{\scriptstyle}="2f";
    (-20,-6)*{\scriptstyle}="1g";
    (0,-6)*{\scriptstyle}="2g";
    (-15,-4)*{\scriptstyle (\sigma, \mathbb{D})};
    (-7,-4)*{\scriptstyle\bullet}="4f";
    (-7,-8)*{\scriptstyle \sigma(n)};
    "1f"; "1g" **\dir{-};
    "1f"; "2f" **\dir{-};
    "1g"; "2g" **\dir{-};
    "2f"; "2g" **\dir{-};
    (12,1)*{\scriptstyle =\,\,\,a\varsigma(\sigma, \mathbb{D})};
    \endxy
    \]
    \caption{Homomorphism $\varsigma$}
	\label{Fig12}
\end{figure}
(2)\, It is proved similarly to (1). There is only one thing to emphasize here,
which is different from the case of type A, that is, $(2p, p)$-decorated parenthesis diagrams
belong to two different cell modules. This leads to the second term of the exact sequence being
a direct sum of $S^+(2p, p)$ and $S^-(2p, p)$.

(3)\, Take the map from $S^+(2p, p)\downarrow$ to $S(n-1, p-1)$ to be the bijection $\beta^+$.
Note that $\beta^+$ always deletes the decoration of the arc $(k, n)$ of a $(2p, p)$-decorated
parenthesis diagram. Then employ Figure \ref{Fig12} again, we get that
$\beta^+$ is a ${\rm TLD}_{n-1}(\delta)$-module homomorphism.
So $S^+(2p, p)\downarrow\, \cong S(2p-1, p-1)$. Similarly, we obtain that $\beta^-$ is an isomorphism
of  ${\rm TLD}_{n-1}(\delta)$-modules from $S^-(2p, p)\downarrow$ to $S(2p-1, p-1)$.
\end{proof}

\medskip

\section{Gram matrices and semi-simplicity}

Let ${\rm TLD}_n(\delta)$ be a Temperley-Lieb algebra of type D with the
cellular basis given in Section 5. In this section, we will give a necessary and sufficient
condition for ${\rm TLD}_n(\delta)$ being semi-simple by  studying Gram matrices of cell modules.

\smallskip

Let $\lam\in\{0^+, 0^-, 1, 2, \cdots, n\}$ and let $(\sigma_1, \mathbb{D}_1), (\sigma_2, \mathbb{D}_2)\in M(\lam)$
be two arbitrary decorated diagrams.
We construct a diagram $\Omega$ by putting the cap diagram of $(\sigma_1, \mathbb{D}_1)$
on the cup diagram of $(\sigma_2, \mathbb{D}_2)$ such that the dots with the same label coincide, then
$\Phi_{\lam}((\sigma_1, \mathbb{D}_1), (\sigma_2, \mathbb{D}_2))$ is completely determined
by the structure of $\Omega$ according to the definition of $\Phi_{\lam}$.

\begin{lem}\label{6.1}  Keep notations as above, then
$$\Phi_{\lam}((\sigma_1, \mathbb{D}_1), (\sigma_2, \mathbb{D}_2))=\left\{
  \begin{array}{ll}
    0, & \hbox{if $\Omega$ contains a closed circuit with odd decorations,}\\
       & \hbox{or the sum of the number of isolated dots and the} \\
       & \hbox{number of non-closed curves less than $\lam (\neq 0^+, 0^-)$};\\
\\
    \delta^{c}, & \hbox{otherwise,}
  \end{array}
\right. $$ where $c$ is the number of circuits in $\Omega$.
\end{lem}

\begin{example} Taking diagrams (3) and (4) of Figure \ref{Fig10}, we get
\begin{figure}[H]
	\[
	\xy
    (30,0)*{\scriptstyle \Omega\,\,\,=};
    (35,0)*{\scriptstyle\circ}="6d";
    (45,0)*{\scriptstyle\circ}="7d";
    (55,0)*{\scriptstyle\circ}="8d";
    (60,2.6)*{\scriptstyle\bullet};
    (65,0)*{\scriptstyle\circ}="9d";
    (75,0)*{\scriptstyle\circ}="10d";
    "6d"; "7d" **\crv{(36,-3.5) & (44,-3.5)};
    "6d"; "7d" **\crv{(36,3.5) & (44,3.5)};
    "8d"; "9d" **\crv{(56,-3.5) & (64,-3.5)};
    "8d"; "9d" **\crv{(56,3.5) & (64,3.5)};
\endxy
    \]
\end{figure}
We have from Lemma \ref{6.1} that the (3, 4)-entry of $G(5, 2)$ is 0
because $\Omega$ contains a circuit with one decoration.

Taking diagrams (8) and (9) of Figure \ref{Fig10}, we get
\begin{figure}[H]
	\[
	\xy
    (30,0)*{\scriptstyle \Omega\,\,\,=};
    (35,0)*{\scriptstyle\circ}="6d";
    (45,0)*{\scriptstyle\circ}="7d";
    (55,0)*{\scriptstyle\circ}="8d";
    (65,0)*{\scriptstyle\circ}="9d";
    (75,0)*{\scriptstyle\circ}="10d";
    "7d"; "8d" **\crv{(46,-3.5) & (54,-3.5)};
    "6d"; "7d" **\crv{(36,3.5) & (44,3.5)};
    "9d"; "10d" **\crv{(66,-3.5) & (74,-3.5)};
    "9d"; "10d" **\crv{(66,3.5) & (74,3.5)};
\endxy
    \]
\end{figure}
Clearly, there is no decorated circuit in $\Omega$. The sum of the number of
isolated dots and the number of non-closed curves is 1, which is equal to $\lambda$.
Then by Lemma \ref{6.1} the (8, 9)-entry of $G(5, 2)$ is $\delta$ because the number of circuits in $\Omega$ is 1.
Similarly, we can obtain all entries of $G(5, 2)$.

$$G(5, 2)=
\left(\begin{array}{cccccccccc}
  \delta^2 & 0 & 0 & 0 & 0 & \delta & \delta & 0 & 1 & 0  \\
  0 & \delta^2 & 0 & 0 & \delta & 0 & \delta & 0 & 1 & \delta  \\
  0 & 0 & \delta^2 & 0 & \delta & 0 & 0 & \delta & 1 & 0  \\
  0 & 0 & 0 & \delta^2 & 0 & \delta & 0 & \delta & 1 & \delta  \\
  0 & \delta & \delta & 0& \delta^2 & 0 & 1 & 1 & \delta & 1  \\
  \delta & 0 & 0 & \delta & 0 & \delta^2 & 1 & 1 & \delta & 1  \\
  \delta & \delta & 0 & 0 & 1 & 1 & \delta^2 & 0 & \delta & 1  \\
  0 & 0 & \delta & \delta & 1 & 1 & 0 & \delta^2 & \delta & 1  \\
  1 & 1 & 1 & 1 & \delta & \delta & \delta & \delta & \delta^2 & \delta  \\
  0 & \delta & 0 & \delta & 1 & 1 & 1 & 1 & \delta  & \delta^2  \\
 \end{array}
\right)
$$
\end{example}

Let us first compute Gram determinants  for some special cases.
Recall that the Chebychev polynomials of the second kind are a sequence of polynomials
with integer coefficients and are determined by the recurrence relation
$$P_0(x)=0 \,\,\,\,P_1(x)=1 \,\,\,\,P_{n+1}(x)=xP_n(x)-P_{n-1}(x)$$
In \cite[Section 5]{W2}, Westbury proved the following result about the
Gram matrices of cell modules of a Temperley-Lieb algebra of type A.
\begin{lem}\cite{W2}\label{6.2}
The determinant of the Gram matrix $\mathcal{G}(n, 1)$ is $P_n(\delta)$.
\end{lem}

Furthermore, Westbury \cite{W2} proved the following recurrence relation on Gram determinants of cell modules.
\begin{lem} \label{as}
If $0<2p<n$, then
$$\det \mathcal{G}(n, p)=
    \det \mathcal{G}(n-1, p)\left(\frac{P_{n-2p+2}}{P_{n-2p+1}}\right)^{{\left(\smallmatrix n-1 \\
p-1 \endsmallmatrix \right)}-{\left(\smallmatrix n-1 \\
p-2 \endsmallmatrix \right)}}\det \mathcal{G}(n-1, p-1).$$
If $n=2p$, then $$\det \mathcal{G}(n, p)=
    \delta^{{\left(\smallmatrix n-1 \\
p-1 \endsmallmatrix \right)}-{\left(\smallmatrix n-1 \\
p-2 \endsmallmatrix \right)}}\det \mathcal{G}(n-1, p-1).$$
\end{lem}

In order to study the Gram matrices of cell modules of ${\rm TLD}_n(\delta)$,
we first generalize Chebychev polynomials to type D.
For $n\geq 3$, let $$Q_2(x)=x^2\,\,\,\,\,Q_3(x)=x^3-2x\,\,\,\,\,\,\,Q_{n+1}(x)=xQ_n(x)-Q_{n-1}(x).$$
We obtain a new sequence of polynomials with integer coefficients,
which will be called Chebychev polynomials of type D.
Then we have a result similar to the case of type A.
\begin{lem}\label{6.3}
The determinant of the Gram matrix $ G(n, 1)$ is $Q_n(\delta)$ for $n\geq 3$.
\end{lem}

\begin{proof}
Note that the cell module $S(n, 1)$ has a basis that consists of $(n, 1)$-decorated parenthesis diagrams.
If these diagrams arranged in the order of Definition \ref{4.5}, the Gram matrix $G(n, 1)$ is
$$
\left(\begin{array}{ccccccc}
  \delta & 0 & 1 & 0 & 0 & \cdots & 0 \\
  0 & \delta & 1 & 0 & 0 & \cdots & 0 \\
  1 & 1 & \delta & 1 & 0 & \cdots & 0 \\
  0 & 0 & 1 & \delta & 1 & \cdots & 0 \\
  \cdots & \cdots & \cdots & \cdots & \cdots & \cdots & \cdots \\
  0 & 0 & \cdots & 0 & 1 & \delta & 1 \\
  0 & 0 & \cdots & 0 & 0 & 1 & \delta \\
 \end{array}
\right)
$$
for $n>2$. Clearly, $\det G(3, 1)=\delta^3-2\delta=Q_3(\delta)$.

Define $G_2=\left(\begin{array}{cc}\delta & 0 \\0 & \delta \end{array}\right)$.
Expanding the determinant of $G(n, 1)$ along the bottom row gives that
$$\det G(n, 1)=\delta\det G(n-1, 1)-\det G(n-2, 1)$$ for $n\geq 4$. That is, the polynomials $\det G(n, 1)$
satisfy the recurrence relation that defines $Q_n(x)$.
\end{proof}

The specific characteristic of the cell module $S(n, 1)$ is that each basis element has only one arc.
Let us consider another two kinds of special cell modules,  $S^+(2p, p)$ and $S^-(2p, p)$.
The basis elements of these cell modules have no isolated dots.
Fortunately, for arbitrary fixed $p$, the Gram matrices of the two cell modules are the same.

\begin{lem}\label{6.4}
$G^+(2p, p)=G^-(2p, p).$
\end{lem}

\begin{proof}
By Lemma \ref{4.8}, we only need to prove that
$$\Phi_{0^+}((\sigma_1, \mathbb{D}_1), (\sigma_2, \mathbb{D}_2))=
\Phi_{0^-}(\alpha(\sigma_1, \mathbb{D}_1), \alpha(\sigma_2, \mathbb{D}_2))$$
for arbitrary $(\sigma_1, \mathbb{D}_1), (\sigma_2, \mathbb{D}_2)\in M(0^+)$.
By Lemma \ref{6.1}, it is enough to compare two graphs $\Omega$ and $\alpha(\Omega)$, where
$\alpha(\Omega)$ is constructed by $\alpha(\sigma_1, \mathbb{D}_1)$ and $\alpha(\sigma_2, \mathbb{D}_2)$.
We split the proof into the following four cases.

{\it Case 1.}\,\, $(\sigma_1(n), n)\notin \mathbb{D}_1$, $(\sigma_2(n), n)\notin \mathbb{D}_2$.
Piecing these two arcs together, we obtain a curve $\Gamma$ without decorations,
which is a part of $\Omega$. On the other hand,
$(\sigma_i(n), n)\in\alpha(\mathbb{D}_i)$, for $i=1, 2$, then piecing the two decorated arcs,
we get a curve $\alpha(\Gamma)$ with two decorations, which is a part of $\alpha(\Omega)$.
According to the rule of removing dots (See Figure \ref{Fig4}),
$\alpha(\Gamma)=\Gamma$, and consequently, $\alpha(\Omega)=\Omega$.

{\it Case 2.}\,\, $(\sigma_1(n), n)\in \mathbb{D}_1$, $(\sigma_2(n), n)\in \mathbb{D}_2$.
This is dual to Case 1 and the proof is similar.  We omit the details.

{\it Case 3.}\,\, $(\sigma_1(n), n)\notin \mathbb{D}_1$, $(\sigma_2(n), n)\in \mathbb{D}_2$.
It can be analyzed similar to Case 1, too. We omit the details here.

{\it Case 4.}\,\, $(\sigma_1(n), n)\in \mathbb{D}_1$, $(\sigma_2(n), n)\notin \mathbb{D}_2$. It is dual to Case 3.
\end{proof}
Note that the above four cases can be illustrated by Figure \ref{Fig13},
in which the left one corresponds to the first two cases and the right one to the others.
\begin{figure}[H]
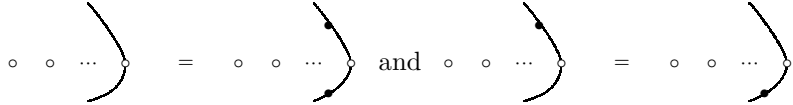

	\[
	\xy
    (5,8)*{\scriptstyle}="1d";
    (5,-5)*{\scriptstyle}="2d";
    (-5,0)*{\scriptstyle\circ};
    (0,0)*{\scriptstyle\circ};
    (5,0)*{\scriptstyle\cdots};
    (10,0)*{\scriptstyle\circ}="1s";
    "1s"; "1d" **\crv{(9.8,2) & (6,7)};
    "1s"; "2d" **\crv{(9.8,-4) & (6,-4.7)};
    (18,0)*{\scriptstyle =};
    (35,8)*{\scriptstyle}="1a";
    (35,-5)*{\scriptstyle}="2a";
    (25,0)*{\scriptstyle\circ};
    (30,0)*{\scriptstyle\circ};
    (35,0)*{\scriptstyle\cdots};
    (40,0)*{\scriptstyle\circ}="1b";
    (37,5)*{\scriptstyle\bullet};
    (37,-4)*{\scriptstyle\bullet};
    "1b"; "1a" **\crv{(39.8,2) & (36,7)};
    "1b"; "2a" **\crv{(39.8,-3) & (36,-4.7)};
    \endxy
    \,\,\,\,\,
    \text{and}\,\,\,\,\,
    \xy
    (5,8)*{\scriptstyle}="1c";
    (5,-5)*{\scriptstyle}="2c";
    (-5,0)*{\scriptstyle\circ};
    (0,0)*{\scriptstyle\circ};
    (5,0)*{\scriptstyle\cdots};
    (10,0)*{\scriptstyle\circ}="1f";
    "1f"; "1c" **\crv{(9.8,2) & (6,7)};
    "1f"; "2c" **\crv{(9.8,-4) & (6,-4.7)};
    (18,0)*{\scriptstyle =};
    (35,8)*{\scriptstyle}="1h";
    (35,-5)*{\scriptstyle}="2h";
    (25,0)*{\scriptstyle\circ};
    (30,0)*{\scriptstyle\circ};
    (35,0)*{\scriptstyle\cdots};
    (40,0)*{\scriptstyle\circ}="1g";
    (7,5)*{\scriptstyle\bullet};
    (37,-4)*{\scriptstyle\bullet};
    "1g"; "1h" **\crv{(39.8,2) & (36,7)};
    "1g"; "2h" **\crv{(39.8,-3) & (36,-4.7)};
    \endxy
    \]
    \caption{Diagrammatic approach}
	\label{Fig13}
\end{figure}
For $S^+(2p, p)$, another necessary preparation is to consider the Gram matrix of its restriction.

\begin{lem}\label{6.5}
$\det G^+(2p, p)=\delta^{\frac{1}{2}\left(\smallmatrix 2p\\
p \endsmallmatrix \right)}\det G(2p-1, p-1)$
\end{lem}

\begin{proof}
We will prove $G^+(2p, p)=\delta G(2p-1, p-1)$ and then the lemma follows from Lemma \ref{4.4}.
Since $\beta^+$ preserves the order by Lemma \ref{4.9}, given two elements $(\sigma_1, \mathbb{D}_1), (\sigma_2, \mathbb{D}_2)\in M^+(0)$,
it is enough to prove
$$\Phi_{0^+}((\sigma_1, \mathbb{D}_1), (\sigma_2, \mathbb{D}_2))=\delta\Phi_1(\beta^+(\sigma_1, \mathbb{D}_1), \beta^+(\sigma_2, \mathbb{D}_2)).$$
Let us compute them through graphs $\Omega$ and $\beta^+(\Omega)$ by Lemma \ref{6.1}. We split the proof into the following three cases.

{\it Case 1.}\,\,In $\Omega$, the number of decorations on the circuit $\rm O_n$ passing through the dot $n$ is odd.
By  Lemma \ref{6.1} and the removing rule in Figure \ref{Fig4},
this implies that $\Phi_{0^+}((\sigma_1, \mathbb{D}_1), (\sigma_2, \mathbb{D}_2))=0$. On the other hand, because $|\mathbb{D}_1|+|\mathbb{D}_2|$ is even,
there is at least one other circuit $\rm O$ in $\Omega$ having odd decorations.
By the definition of $\beta^+$, the decorations on $\rm O$ is not changed by $\beta^+$, that is, $\rm O$ is a circuit in $\beta^+(\Omega)$ too.
As a result, $\Phi_1(\beta^+(\sigma_1, \mathbb{D}_1), \beta^+(\sigma_2, \mathbb{D}_2))=0$.

{\it Case 2.}\,\,In $\Omega$, the number of decorations on the circuit $\rm O_n$ is even, but there is one other circuit $\rm O$ in $\Omega$ with odd decorations. Clearly, $\beta^+$ does not change the circuit $\rm O$ and consequently, both $\Phi_{0^+}((\sigma_1, \mathbb{D}_1), (\sigma_2, \mathbb{D}_2))$ and
$\Phi_1(\beta^+(\sigma_1, \mathbb{D}_1), \beta^+(\sigma_2, \mathbb{D}_2))$ are zeros.

{\it Case 3.}\,\,The graph $\Omega$ contains no circuit with odd decorations. If the number of circuits in $\Omega$ is $c$,
then by the definition of $\beta^+$, the number of circuits in $\beta^+(\Omega)$ is $c-1$ clearly. It follows from Lemma \ref{6.1} that
the equality holds.
\end{proof}

Combing Lemma \ref{4.9}, \ref{6.4} with Lemma \ref{6.5} yields the following corollary.

\begin{cor}\label{6.6}
$\det G^-(2p, p)=\delta^{\frac{1}{2}\left(\smallmatrix 2p\\
p \endsmallmatrix \right)} \det G(2p-1, p-1)$.
\end{cor}

For the sake of description, we define a matrix $G(2p, p)$, which is related with $G^+(2p, p)$ and $G^-(2p, p)$.
Given two $(2p, p)$-decorated parenthesis diagrams
$(\sigma_1, \mathbb{D}_1)$ and $(\sigma_2, \mathbb{D}_2)$, define the corresponding matrix element to be
$\Phi_1(\iota(\sigma_1, \mathbb{D}_1), \iota(\sigma_2, \mathbb{D}_2))$, where $\iota$ is defined in the proof of Lemma \ref{5.2}.
Is is helpful to point out that
the matrix $G(2p, p)$ is not the Gram matrix of any cell module.
Clearly, if $|\mathbb{D}_1|$ is odd and $|\mathbb{D}_2|$ is even, then $\Phi_1(\iota(\sigma_1, \mathbb{D}_1), \iota(\sigma_2, \mathbb{D}_2))=0$.
This implies that if we adjust the order appropriately, the matrix $G(2p, p)$ becomes $\left(\begin{array}{cc} G^+(2p, p) & 0\\
0 & G^-(2p, p) \end{array} \right)$. As a result, $\det G(2p, p)=(\det G^+(2p, p))^2$.

Before we can study the recurrence relation of determinants of Gram matrices of cell modules,
we also need to investigate some more information about the internal structure of a Gram matrix.

\begin{lem}\label{6.7}
Let $n>2p$ and let $\left(\begin{array}{cc} G(n, p)_{11} & G(n, p)_{12}\\
G(n, p)_{21} & G(n, p)_{22}\end{array} \right)$ be the block form of Gram matrix $G(n, p)$,
where $G(n, p)_{11}$ is a $\left(\smallmatrix n-1 \\
p \endsmallmatrix \right)\times\left(\smallmatrix n-1 \\
p \endsmallmatrix \right)$ matrix. Then we have
\begin{enumerate}
\item[(1)]\, $G(n, p)_{12}'=G(n, p)_{21},  \,\,\,\,\,\,\,G(n, p)_{11}=G(n-1, p);$
\item[(2)]\, $G(n, p)_{12}$ and $G(n, p)_{22}$ having the block form
$$\left(\begin{array}{cc} 0 & \star\\
G(n-2, p-1) & \star\end{array} \right)\,\,\, and \,\,\,\left(\begin{array}{cc} \delta G(n-2, p-1) & \star\\
\star & \star\end{array} \right),$$ respectively.
\end{enumerate}
\end{lem}

\begin{proof}
(1) Note that $G(n, p)$ is symmetric and then $G(n, p)_{12}'=G(n, p)_{21}$ follows.
According to Definition \ref{4.5},  the diagrams with $n$ being an isolated dot are the smallest terms.
The map $\iota$ defined in Lemma \ref{5.2} clearly preserves the order
and the values taken by the bilinear form, so $G(n, p)_{11}=G(n-1, p)$.

(2) By Definition \ref{4.5}, the diagrams next to those with $n$ being an isolated dot have $(n-1, n)$ as an arc.
Given two diagrams with $n$ being isolated, the diagram with $n-1$ isolated comes first.
For $S=(\sigma_1, \mathbb{D}_1),\,\, T=(\sigma_2, \,\mathbb{D}_2)$, if $\sigma_1(n-1)=n-1$,
$\sigma_1(n)=n$ and $\sigma_2(n-1)=n$, then we have from Lemma \ref{6.1} that the value taken by the bilinear form is zero.
This implies that the left top block of $G(n, p)_{12}$ is zero.
If $\sigma_1(n)=n$, $\sigma_1(n-1)\neq n-1$ and $\sigma_2(n-1)=n$, then Lemma \ref{6.1} gives
that the value taken by the bilinear form does not change if we remove the last two dots in both diagrams,
that is, the left bottom block of $G(n, p)_{12}$ is $G(n-2, p-1)$.

In order to  handle the block form of $G(n, p)_{22}$, we only need to consider the diagrams with $(n-1, n)$ being an arc.
We claim that $(n-1, n)$ is not decorated. In fact, if $(n-1, n)$ decorated,
then by Definition \ref{4.2} the diagram has no isolated dots and thus $n=2p$.
It is in contradiction with $n>2p$. With this preparation, it follows from Lemma \ref{6.1}
that the left top block of $G(n, p)_{22}$ is $\delta G(n-2, p-1)$.
\end{proof}

\begin{remark}
For convenience of the read we explain which diagrams correspond to which
blocks of the matrices in Lemma \ref{6.7} in detail by the form below as an example.
For simplicity we will not explain other block matrices that appear in the paper and the reader can analyze them similarly.

\begin{figure}[H]
	\[
	\xy
    (-40,-20)*{\scriptstyle}="1a";
    (-40,65)*{\scriptstyle}="2a";
    "1a"; "2a" **\dir{-};
    (40,-20)*{\scriptstyle}="1b";
    (40,65)*{\scriptstyle}="2b";
    "1b"; "2b" **\dir{-};
    (-55,15)*{\scriptstyle}="1c";
    (40,15)*{\scriptstyle}="2c";
    "1c"; "2c" **\dir{-};
    (-55,55)*{\scriptstyle}="1d";
    (40,55)*{\scriptstyle}="2d";
    "1d"; "2d" **\dir{-};
    (0,-20)*{\scriptstyle}="1e";
    (0,65)*{\scriptstyle}="2e";
    "1e"; "2e" **\dir{-};
    (-65,35)*{\scriptstyle n\,\, isolated};
    (-65,-3)*{\scriptstyle n\,\, not};
    (-65,-6)*{\scriptstyle isolated};
    (-20,65)*{\scriptstyle n\,\, isolated};
    (20,65)*{\scriptstyle n\,\,not\,\, isolated};
    (20,-20)*{\scriptstyle}="1f";
    (20,60)*{\scriptstyle}="2f";
    "1f"; "2f" **\dir{-};
    (-55,-5)*{\scriptstyle}="1g";
    (40,-5)*{\scriptstyle}="2g";
    "1g"; "2g" **\dir{-};
    (0,35)*{\scriptstyle}="1h";
    (40,35)*{\scriptstyle}="2h";
    "1h"; "2h" **\dir{-};
    (-55,35)*{\scriptstyle}="1j";
    (-40,35)*{\scriptstyle}="2j";
    "1j"; "2j" **\dir{-};
    (-20,-20)*{\scriptstyle}="1k";
    (-20,15)*{\scriptstyle}="2k";
    "1k"; "2k" **\dir{-};
    (-20,55)*{\scriptstyle}="1l";
    (-20,60)*{\scriptstyle}="2l";
    "1l"; "2l" **\dir{-};
    (-20,35)*{\scriptstyle G(n, p)_{11}};
    (10,45)*{\scriptstyle 0};
    (10,25)*{\scriptstyle G(n-2, p-1)};
    (30,45)*{\scriptstyle \star};
    (30,25)*{\scriptstyle \star};
    (10,5)*{\scriptstyle \delta G(n-2, p-1)};
    (30,5)*{\scriptstyle \star};
    (30,-13)*{\scriptstyle \star};
    (10,-13)*{\scriptstyle \star};
    (-10,5)*{\scriptstyle G'(n-2, p-1)};
    (-30,5)*{\scriptstyle 0};
    (-30,-13)*{\scriptstyle \star};
    (-10,-13)*{\scriptstyle \star};
    (-30,58)*{\scriptstyle n-1\,\, isolated};
    (-50,45)*{\scriptstyle n-1\,\, isolated};
    (-10,58)*{\scriptstyle others};
    (-50,25)*{\scriptstyle others};
    (5,60)*{\scriptstyle\circ}="1s";
    (15,60)*{\scriptstyle\circ}="2s";
    (5,62)*{\scriptstyle n-1};
    (15,62)*{\scriptstyle n};
    "1s"; "2s" **\crv{(6,56.5) & (14,56.5)};
    (30,58)*{\scriptstyle others};
    (-55,5)*{\scriptstyle\circ}="1t";
    (-45,5)*{\scriptstyle\circ}="2t";
    (-55,7)*{\scriptstyle n-1};
    (-45,7)*{\scriptstyle n};
    "1t"; "2t" **\crv{(-54,1.5) & (-46,1.5)};
    (-50,-13)*{\scriptstyle others};
    \endxy
    \]
\end{figure}
\end{remark}

Note that some Gram matrices are the same as those of cell modules of type A up to multiplication with $\delta$. The following result is simple.

\begin{lem}\label{6.8}
If $\lam=\dot{k}$, then $\dot{G}(n, p)=\delta \mathcal{G}(n, p)$ for all $n\geq 2$. Moreover, $\dot{G}(2p, p)=\delta \dot{G}(2p-1, p-1)$.
\end{lem}

\begin{proof}
Since $\lam=\dot{k}$, for arbitrary $S, T\in M(\dot{k})$, both the diagrams
$C_{S, S}^{\dot{k}}$ and $C_{T, T}^{\dot{k}}$ are of the first type.  By the rules
in Figure \ref{Fig4} on removing circuits and dots, two decorated circuits are replaced
by one with a coefficient $\delta$, and this complete the proof of $\dot{G}(n, p)=\delta \mathcal{G}(n, p)$.
Combing the first equality with Lemma \ref{as} leads to the second one.
\end{proof}

A direct corollary of this lemma is about the quasi-heredity of ${\rm TLD}_n(\delta)$.

\begin{cor}\label{6.9}
The algebra ${\rm TLD}_n(\delta)$ {\rm($n\geq 2$)} is quasi-hereditary if and only if $\delta\neq 0$.
\end{cor}

\begin{proof}
If $\delta=0$, then Lemma \ref{6.8} implies that $\Phi_{\dot{k}}=0$ for all $k=n-2, n-4 \cdots$
and thus ${\rm TLD}_n(\delta)$ is not quasi-hereditary.

On the other hand, suppose that $\delta\neq 0$. If $n$ is odd,
then $$(e_1e_3\cdots e_p)^2=\delta^p(e_1e_3\cdots e_p).$$ This gives that $\Phi_{n-2p}=\delta^p\neq 0$ for $p\geq 1$.
Moreover, $$(e_{\bar{1}}e_1e_3\cdots e_p)^2=\delta^{p+1}(e_{\bar{1}}e_1e_3\cdots e_p),$$
that is, $\Phi_{\dot{n-2p}}=\delta^{p+1}\neq 0$. Consequently, $\Phi_{\lam}\neq 0$ for all $\lam\in\Lambda$, that is,
${\rm TLD}_n(\delta)$ is quasi-hereditary.

If $n$ is even, the only extra cases to consider are $\lam =0^+,\,0^-,\,\dot{0}$.
For $0^+$, we have $(e_1e_3\cdots e_p)^2=\delta^p(e_1e_3\cdots e_p)$, for $0^-$,
$(e_{\bar{1}}e_3\cdots e_p)^2=\delta^p(e_{\bar{1}}e_3\cdots e_p)$ and for $\dot{0}$,
$(e_{\bar{1}}e_1e_3\cdots e_p)^2=\delta^{p+1}(e_{\bar{1}}e_1e_3\cdots e_p)$, where $p=\frac{n}{2}$.
We deduce that all $\Phi_{\lam}\neq 0$ for $\lam\in\Lambda$.
\end{proof}

To give the recurrence relation of the determinants of the Gram matrices,
we temporarily let $R$ be a field with $Char R\neq 2$ and $\delta$ an indeterminate.
Let ${\rm TLD}_{n}(\delta)$ the corresponding Temperley-Lieb
algebra of type D over $R[\delta]$. Our strategy is to embed the algebra to a semi-simple one.
Let $q$ be another indeterminate over $R$ and denote by $F$ the field of fractions of $R[q]$.
Define an $R$-linear map from $R[\delta]$ to $F$ which send $\delta$ to $q+q^{-1}$. By using this map we
specialize ${\rm TLD}_{n}(\delta)$ to $F$  and denote the algebra by ${\rm TLD}_{F, n}(\delta)$.
We claim that ${\rm TLD}_{F, n}(\delta)$ is semi-simple. In fact, let $H_n(q^2)$ be the Hecke algebra of type $D_n$ over $F$.
Then by the results of \cite{H,P}, $H_n(q^2)$ is semi-simple. Note that Fan proved in \cite{F1}
that ${\rm TLD}_{F, n}(\delta)$ is isomorphic to a quotient algebra of $H_n(q^2)$ by certain ideal $I$. As a result,
${\rm TLD}_{F, n}(\delta)$ is semi-simple.

Moreover, by cellular theory,  ${\rm TLD}_{F, n}(\delta)$ is still a cellular algebra.
The notations of cell modules and Gram matrices will
not be changed because there is no danger of confusion. All of the cell modules of ${\rm TLD}_{F, n}(\delta)$
are a complete set of pairwise non-isomorphic absolutely irreducible modules.

\begin{lem}\label{6.10}
If $n\geq2p+1$, then the determinant of the Gram matrices satisfy the following recurrence relation
$$\det G(n, p)=
    \det G(n-1, p)r(n, p)^{\left(\smallmatrix n-1 \\
p-1 \endsmallmatrix \right)}\det G(n-1, p-1),$$
where $r(n, p)\in F$ satisfies $r(n, p)=\delta-\frac{1}{r(n-1, p)}$.
\end{lem}

\begin{proof}
Since the algebra ${\rm TLD}_{F, n}(\delta)$ is semi-simple,
the exact sequences of Lemma \ref{5.2} split. Then there is a unique isomorphism
$$\Theta: S(n, p)\downarrow\rightarrow S(n-1, p)\bigoplus S(n-1, p-1)$$ such
that the matrix of $\Theta$ is of the form
$V(n, p)=\left(\begin{array}{cc} E & X(n, p) \\ 0 & E \\ \end{array}\right)$.
Take the transpose of its inverse and denote the matrix obtained by $W(n, p)$.
Then $W(n, p)$ converts the bilinear form of $S(n, p)$ to a bilinear form $\phi$ of $S(n-1, p)\bigoplus S(n-1, p-1)$
satisfying the property of Lemma \ref{bl}. Consquently, the Gram matrix of $\phi$ is a block diagonal one,
in which the (1, 1)-block is $G(n-1, p)$ and the (2, 2)-block is $r(n, p)G(n-1, p-1)$, where $r(n, p)\in F$. Writing in matrix form, we have
$$
G(n, p)=V(n, p)'\left(\begin{array}{cc} G(n-1, p) & 0 \\ 0 & r(n, p)G(n-1, p-1) \\ \end{array}\right)V(n, p).
$$
This gives
\begin{align}
G(n, p)_{12}=G(n-1, p)X(n, p)
\end{align} and
\begin{align}
G(n, p)_{22}&=X(n, p)'G(n-1, p)X(n, p)+r(n, p)G(n-1, p-1) \\
&=X(n, p)'G(n, p)_{12}+r(n, p)G(n-1, p-1).
\end{align}

Writing $X(n, p)$ in block form $\left(\begin{array}{cc} X(n, p)_{11} & \star \\ X(n, p)_{21} & \star \\ \end{array}\right)$ and
substituting it and that in Lemma \ref{6.7} into (6.1) yields
\begin{align}
G(n-2, p)X(n, p)_{11}+G(n-1, p)_{12}X(n, p)_{21}=0
\end{align}
and
\begin{align}
G(n-2, p-1)=G(n-1, p)_{21}X(n, p)_{11}+G(n-1, p)_{22}X(n, p)_{21}.
\end{align}
Replace $n$ by $n-1$ in (6.1) and then substitute it into (6.4). Then
\begin{align}
X(n, p)_{11}+X(n-1, p)X(n, p)_{21}=0.
\end{align}
Writing (6.2) in $n-1$ version and substituting it into (6.5) leads to
\begin{align*}
G(n-2, p-1)&=G(n-1, p)_{21}X(n, p)_{11}+\\& X(n-1, p)'G(n-2, p)X(n-1, p)X(n, p)_{21}+\\& r(n-1, p)G(n-2, p-1)X(n, p)_{21}.
\end{align*}
Combining (6.6) with the above equality yields
\begin{align}
X(n, p)_{21}=\frac{1}{r(n-1, p)}E.
\end{align}
Replace the matrices of (6.3) by their block forms. Then a direct computation leads to
$X(n, p)_{21}'=(\delta-r(n, p))E$. By taking transpose on both sides of this equality, we get
\begin{align}
X(n, p)_{21}=(\delta-r(n, p))E.
\end{align}
Now we arrive at the recurrence relation of $r(n, p)$ by combining (6.7) with (6.8).
\end{proof}

\begin{remark}
The technique of obtaining the recurrence relation of $r(n, p)$ in Lemma \ref{6.10} is similar to the case of type A used in \cite{RS}.
We write it here for completeness.
\end{remark}

Now let us determine these $r(n, p)$ by computing the starting point for the recursion.

\begin{lem}\label{6.13}
Let $n\geq 2p+1$. Then $r(n, p)=\frac{Q_{n-2p+2}}{Q_{n-2p+1}}$
\end{lem}

\begin{proof}
Let $n=2p+1$. By Lemma \ref{6.7} and \ref{6.10}, there is a unique matrix $W(2p+1, p)$ such that
\begin{align}
W(2p+1, p)G(2p+1, p)W(2p+1, p)'=\left(\begin{array}{cc} G(2p, p) & 0 \\ 0 & r(2p+1, p)G(2p, p-1) \end{array} \right),
\end{align}
where $G(2p, p-1)$ has the block form $\left(\begin{array}{cc} G(2p-1, p-1) & \star \\ \star & \star \\ \end{array}\right)$.

On the other hand, by the definition of $G(2p, p)$, there exists a unique matrix
$P$ with $P^2=E$ such that the left top block of the matrix $PG(2p+1, p)P'$
being equal to $\left(\begin{array}{cc} G^+(2p, p) & 0 \\0 & G^-(2p, p)\end{array}\right)$,
or $\left(\begin{array}{cc} \delta G(2p-1, p-1) & 0 \\0 & \delta G(2p-1, p-1) \end{array}\right)$
by Lemma \ref{6.5} and Corollary \ref{6.6}.
Analyzing the left bottom block of $PG(2p+1, p)P'$ by a process similar
to that in Lemma \ref{6.7} shows that the block has the block form
$\left(\begin{array}{cc} G(2p-1, p-1) & G(2p-1, p-1) \\\star & \star\end{array}\right)$.
Now employing Lemma \ref{6.7} again, we  receive that
the right bottom block of $PG(2p+1, p)P'$ having the block form
$\left(\begin{array}{cc} \delta G(2p-1, p-1) & \star \\\star & \star\end{array}\right)$.
It is necessary to point out that $$W(2p+1, p)G(2p+1, p)W(2p+1, p)'=W(2p+1, p)PPG(2p+1, p)P'P'W(2p+1, p)'.$$
Then the equality (6.9) forces the left bottom block of $W(n, p)P$ has the block form
$\left(\begin{array}{cc}  -\frac{1}{\delta}E & -\frac{1}{\delta}E \\\star & \star\end{array}\right)$.
Substituting the block forms of matrices
$W(2p+1, p)P$ and $PG(2p+1, p)P'$ into (6.9) yields
$$r(2p+1, p)=\delta-\frac{2}{\delta}=\frac{\delta^3-2\delta}{\delta^2}=\frac{Q_3}{Q_2}.$$
This gives the starting point for the recursion.

Now suppose that $r(n, p)=\frac{Q_{n-2p+2}}{Q_{n-2p+1}}$.
Then by Lemma \ref{6.10} and the definition of the polynomial sequence $Q_k$, $$r(n+1, p)=\delta-\frac{Q_{n-2p+1}}{Q_{n-2p+2}}
=\frac{\delta Q_{n-2p+2}-Q_{n-2p+1}}{Q_{n-2p+2}}
=\frac{Q_{n-2p+3}}{Q_{n-2p+2}}.$$
We have completed the proof.
\end{proof}

\begin{cor}\label{6.14}
If $n\geq2p+1$, then the determinants of the Gram matrices satisfy the following recurrence relation
$$\det G(n, p)=
    \det G(n-1, p)\left(\frac{Q_{n-2p+2}}{Q_{n-2p+1}}\right)^{\left(\smallmatrix n-1 \\
p-1 \endsmallmatrix \right)}\det G(n-1, p-1).$$
\end{cor}

Furthermore, we want to give a closed formula for the Gram-determinant. To this end, we need the following lemma.

\begin{lem}\label{6.15}
For all $n\geq2p$,

$$\prod_{r=0}^{p-1}\left(\frac{Q_{n-p-r+1}}{Q_{p-r}}\right)^{\left(\smallmatrix n \\
r \endsmallmatrix \right)}=\prod_{r=0}^{p-1}\left(\frac{Q_{n-p-r}}{Q_{p-r}}\right)^{\left(\smallmatrix n-1 \\
r \endsmallmatrix \right)}\left(\frac{Q_{n-2p+2}}{Q_{n-2p+1}}\right)^{\left(\smallmatrix n-1 \\
p-1 \endsmallmatrix \right)}\prod_{r=0}^{p-2}\left(\frac{Q_{n-p-r+1}}{Q_{p-r-1}}\right)^{\left(\smallmatrix n-1 \\
r \endsmallmatrix \right)}$$
\end{lem}

\begin{proof}
The equality is proved by a direct computation. The following calculation shows that the denominators agree.
\begin{eqnarray*}
& &\prod_{r=0}^{p-1}{Q_{p-r}}^{\left(\smallmatrix n-1 \\
r \endsmallmatrix \right)}\prod_{r=0}^{p-2}{Q_{p-r-1}}^{\left(\smallmatrix n-1 \\
r \endsmallmatrix \right)}\\ &=&
\prod_{r=0}^{p-1}{Q_{p-r}}^{\left(\smallmatrix n-1 \\
r \endsmallmatrix \right)}\prod_{r=1}^{p-1}{Q_{p-r}}^{\left(\smallmatrix n-1 \\
r-1 \endsmallmatrix \right)}\\ &=&
Q_p\prod_{r=1}^{p-1}{Q_{p-r}}^{\left(\smallmatrix n \\
r \endsmallmatrix \right)}\\ &=&
\prod_{r=0}^{p-1}{Q_{p-r}}^{\left(\smallmatrix n \\
r \endsmallmatrix \right)}
\end{eqnarray*}

And the following calculation shows that the numerators agree.

\begin{eqnarray*}
& &\prod_{r=0}^{p-1}{Q_{n-p-r}}^{\left(\smallmatrix n-1 \\
r \endsmallmatrix \right)}\left(\frac{Q_{n-2p+2}}{Q_{n-2p+1}}\right)^{\left(\smallmatrix n-1 \\
p-1 \endsmallmatrix \right)}\prod_{r=0}^{p-2}{Q_{n-p-r+1}}^{\left(\smallmatrix n-1 \\
r \endsmallmatrix \right)}\\
&=&{Q_{n-2p+1}}^{\left(\smallmatrix n-1 \\
p-1 \endsmallmatrix \right)}\prod_{r=1}^{p-1}{Q_{n-p-r+1}}^{\left(\smallmatrix n-1 \\
r-1 \endsmallmatrix \right)}\left(\frac{Q_{n-2p+2}}{Q_{n-2p+1}}\right)^{\left(\smallmatrix n-1 \\
p-1 \endsmallmatrix \right)}\prod_{r=0}^{p-2}{Q_{n-p-r+1}}^{\left(\smallmatrix n-1 \\
r \endsmallmatrix \right)}\\
&=&{Q_{n-2p+1}}^{\left(\smallmatrix n-1 \\
p-1 \endsmallmatrix \right)}\left(\frac{Q_{n-2p+2}}{Q_{n-2p+1}}\right)^{\left(\smallmatrix n-1 \\
p-1 \endsmallmatrix \right)}\prod_{r=1}^{p-2}{Q_{n-p-r+1}}^{\left(\smallmatrix n-1 \\
r-1 \endsmallmatrix \right)}\times\\
& &\prod_{r=1}^{p-2}{Q_{n-p-r+1}}^{\left(\smallmatrix n-1 \\
r \endsmallmatrix \right)}{Q_{n-2p+2}}^{\left(\smallmatrix n-1 \\
p-2 \endsmallmatrix \right)}Q_{n-p+1}\\
&=&Q_{n-p+1}\prod_{r=1}^{p-2}{Q_{n-p-r+1}}^{\left(\smallmatrix n \\
r \endsmallmatrix \right)}{Q_{n-2p+2}}^{\left(\smallmatrix n \\
p-1 \endsmallmatrix \right)}\\
&=&\prod_{r=0}^{p-1}{Q_{n-p-r+1}}^{\left(\smallmatrix n \\
r \endsmallmatrix \right)}
\end{eqnarray*}
Then we complete the proof.
\end{proof}

By Lemma \ref{6.15}, we can obtain the following proposition by a direct computation. We omit the details here.

\begin{prop}\label{6.16}
If $n\geq2p$, then

$$\det G(n, p)=
   \prod_{r=0}^{p-1}\left(\frac{Q_{n-p-r+1}}{Q_{p-r}}\right)^{\left(\smallmatrix n \\
r \endsmallmatrix \right)}\prod_{s=1}^{p-1}\left[\prod_{r=0}^{p-s-1}\left(\frac{Q_{p-s-r+2}}{Q_{p-s-r}}\right)^{\left(\smallmatrix 2\left(p-s\right)+1\\
r \endsmallmatrix \right)}\right]^{\left(\smallmatrix n-2(p-s+1)  \\
s-1 \endsmallmatrix \right)}$$
\end{prop}

\begin{example}
Using the formula above, we have
\begin{eqnarray*}
\det G(8, 3)&=&
   \prod_{r=0}^{2}\left(\frac{Q_{6-r}}{Q_{3-r}}\right)^{\left(\smallmatrix 8 \\
r \endsmallmatrix \right)}\prod_{s=1}^{2}\left[\prod_{r=0}^{2-s}\left(\frac{Q_{5-s-r}}{Q_{3-s-r}}\right)^{\left(\smallmatrix 2\left(3-s\right)+1\\
r \endsmallmatrix \right)}\right]^{\left(\smallmatrix 2+2\left(s-1\right)  \\
s-1 \endsmallmatrix \right)}\\
&=&\prod_{r=0}^{2}\left(\frac{Q_{6-r}}{Q_{3-r}}\right)^{\left(\smallmatrix 8 \\
r \endsmallmatrix \right)}\left[\prod_{r=0}^{1}\left(\frac{Q_{4-r}}{Q_{2-r}}\right)^{\left(\smallmatrix 5 \\
r \endsmallmatrix \right)}\right]^{\left(\smallmatrix 2  \\
0 \endsmallmatrix \right)}\left[\prod_{r=0}^{0}\left(\frac{Q_{3-r}}{Q_{1-r}}\right)^{\left(\smallmatrix 3 \\
r \endsmallmatrix \right)}\right]^{\left(\smallmatrix 4  \\
1 \endsmallmatrix \right)}\\
&=&\frac{Q_{3}^{8}Q_{4}^{29}Q_{5}^{8}Q_{6}}{Q_{2}^{9}}\\
&=&\delta^{58}\left(\delta^2-2\right)^{8}\left(\delta^2-3\right)^{29}\left(\delta^4-4\delta^2+2\right)^{8}\left(\delta^4-5\delta^2+5\right)
\end{eqnarray*}
\end{example}

\smallskip

Now we are in a position to give the main result of this paper.

\begin{thm}\label{6.17}
Let $R$ be a field with $Char R\neq 2$, $\delta\in R$ and $n\geq 4$ a positive integer.
Let ${\rm TLD}_n(\delta)$ be a Temperley-Lieb algebra of type D over $R$.
Then ${\rm TLD}_n(\delta)$ is semi-simple if and only if  $P_s(\delta)\neq 0$ for $1<s\leq n$ and $Q_t(\delta)\neq 0$ for $2<t\leq n$.
\end{thm}

\begin{proof}
If ${\rm TLD}_n(\delta)$ is semi-simple, then by Lemma \ref{2.4} each cell module is simple and
the determinant of its Gram matrix is nonzero. Of course, $\det G(n, 1)\neq 0$ and $\det \mathcal{G}(n, 1)\neq 0$.
By Lemmas \ref{6.2}, \ref{6.3} and \ref{6.8}
this implies that $\delta\neq 0$, $P_n(\delta)\neq 0$ and $Q_n(\delta)\neq 0$.
Furthermore, computing $\det G(n, 1)$ by Corollary \ref{6.14} shows $\det G(s, 1)\neq 0$ for $s=3, \cdots, n-1$,
that is, $Q_t(\delta)\neq 0$ for $t=3, \cdots, n-1$. Similarly, we achieve $P_s(\delta)\neq 0$ for $s=3, \cdots, n-1$.

On the contrary, if $P_s(\delta)\neq 0$ for $1<s\leq n$ and $Q_t(\delta)\neq 0$ for $2<t\leq n$,
then Lemmas 6.2-6.6, \ref{6.8} and Proposition \ref{6.16} show that all the determinants of
Gram matrices of cell modules are nonzero, that is, the algebra ${\rm TLD}_n(\delta)$ is semi-simple by Lemma \ref{2.4}.
\end{proof}

\medskip

\section{Forked Temperley-Lieb algebras}

In order to study intermediate subfactors, Grossman \cite{G3} introduced the
so-called forked Temperley-Lieb algebras. In this section, we investigate the quasi-heredity,
semi-simplicity of these algebras.

In fact, if we add a relation $e_{\bar{1}}e_1=0$ in Definition \ref{3.1}, then we get a quotient
algebra of  ${\rm TLD}_n(\delta)$, which is the forked Temperley-Lieb algebra.
We refer the reader to \cite{G3} for details. Denote the algebra by ${\rm FTL}_n(\delta)$.
Let us first consider the graphical basis of it. In fact, $e_{\bar{1}}e_1$ is the following diagram

\begin{figure}[H]
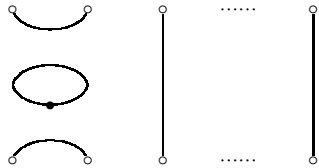

    \[
    \xy
    (35,0)*{\scriptstyle}="1b";
    (45,0)*{\scriptstyle}="2b";
    (35,10)*{\scriptstyle\circ}="1f";
    (45,10)*{\scriptstyle\circ}="2f";
    (55,10)*{\scriptstyle\circ}="3f";
    (65,10)*{\scriptstyle\cdots\cdots}="4f";
    (75,10)*{\scriptstyle\circ}="5f";
    (35,-10)*{\scriptstyle\circ}="1c";
    (45,-10)*{\scriptstyle\circ}="2c";
    (55,-10)*{\scriptstyle\circ}="3c";
    (65,-10)*{\scriptstyle\cdots\cdots}="4c";
    (75,-10)*{\scriptstyle\circ}="5c";
    (40,-2.8)*{\scriptstyle\bullet};
    "1b"; "2b" **\crv{(36,3.5) & (44,3.5)};
    "1b"; "2b" **\crv{(36,-3.5) & (44,-3.5)};
    "1f"; "2f" **\crv{(36,6.5) & (44,6.5)};
    "5c"; "5f" **\dir{-};
    "1c"; "2c" **\crv{(36,-6.5) & (44,-6.5)};
    "3f"; "3c" **\dir{-};
    \endxy
    \]
    \caption{$e_{\bar{1}}e_1$}
	\label{Fig14}
\end{figure}
\noindent Thus the ideal generated by $e_{\bar{1}}e_1$ has a basis, which consists of all the first type diagrams in Lemma \ref{3.3}.
Consequently, thinking of ${\rm FTL}_n(\delta)$ as a diagram algebra, all the second type diagrams form a basis of it.

\begin{lem}
$\dim {\rm FTL}_n(\delta)=\frac{1}{2}\left(\begin{array}{c}2n \\n \\\end{array}\right)$
\end{lem}

Employing the result on ${\rm TLD}_n(\delta)$, one can obtain the cellular structure easily.
We omit the details and left it to the reader. Now we study the quasi-heredity of ${\rm TLD}_n(\delta)$.

\begin{prop}
Let $K$ be a field, $\delta\in K$ and ${\rm FTL}_n(\delta)$ a forked Temprley-Lieb algebra over $K$.
If $\delta\neq 0$, then ${\rm FTL}_n(\delta)$ is quasi-hereditary. If $\delta=0$, then
${\rm FTL}_n(\delta)$ is quasi-hereditary if and only if $n$ is odd.
\end{prop}

\begin{proof}
Let $\delta\neq 0$. For each $\lam\in\Lambda$, take a decorated parenthesis diagram $S\in M(\lam)$.
Then by Lemma \ref{6.1} $\Phi_{\lam}(S, S)\neq 0$ and consequently, ${\rm FTL}_n(\delta)$ is quasi-hereditary.

Let $\delta=0$. If $n$ is even, then clearly $\Phi_{0^+}=0$, $\Phi_{0^-}=0$ and
thus ${\rm FTL}_n(\delta)$ is not quasi-hereditary. If $n$ is odd, for cell modules $S(n, p)$,
take $S=(12)(34)\cdots(2p-1,2p)(2p+1)\cdots(n)$ and $T=(23)(45)\cdots(2p, 2p+1)(2p+2)\cdots(n)$.
Then it is easy to check $\Phi(S, T)=1$, that is, ${\rm FTL}_n(\delta)$ is quasi-hereditary.
\end{proof}

The following semi-simplicity criterion is a direct corollary of Theorem \ref{6.17} and we omit the proof.

\begin{prop}
Let $K$ be a field with $Char K\neq 2$, $0\neq\delta\in K$ and ${\rm FTL}_n(\delta)$ a forked Temprley-Lieb algebra over $K$.
For $n\geq 3$, ${\rm FTL}_n(\delta)$ is semi-simple if and only if $Q_t(\delta)\neq 0$, where $t\in \{2, \cdots, n\}$.
\end{prop}

\begin{remark}
In \cite{LS}, Lejczyk and Stroppel introduced a quotient algebra $\widehat{TL}(\mathcal{W})$ of ${\rm TLD}_n(\delta)$.
In case of $n$ being odd, $\widehat{TL}(\mathcal{W})$ is in fact ${\rm FTL}_n(\delta)$. In case of $n$ being even,
$\widehat{TL}(\mathcal{W})$ is a quotient algebra of ${\rm FTL}_n(\delta)$.
\end{remark}

\bigskip\bigskip

\noindent{\bf Acknowledgement} The authors would like to express their
sincere thanks to the
anonymous referee for her/his numerous helpful comments and corrections.
The authors are grateful to Professor Shoumin Liu and Doctor Pei Wang for some discussions about this topic.
Part of this work was done when Li visited School of Mathematics Sciences at Hebei Normal University in the summer of 2019.
He takes this opportunity to
express his sincere thanks to the School of Mathematics Sciences for the hospitality during his visit.

\end{document}